\documentclass[a4paper,11pt]{amsart}
\usepackage[top=28mm,right=27mm,bottom=28mm,left=27mm]{geometry}
\usepackage{amssymb}
\usepackage{amsmath}
\usepackage{amsfonts}
\usepackage{url}
\newcommand{\PSp}{\mathrm{PSp}}
\newcommand{\PGL}{\mathrm{PGL}}
\newcommand{\Sp}{\mathrm{Sp}}
\newcommand{\SO}{\mathrm{SO}}
\newcommand{\Aut}{\mathrm{Aut}}
\newcommand{\Inndiag}{\mathrm{Inndiag}}
\newcommand{\diag}{\mathrm{diag}}
\newcommand{\SL}{\mathrm{SL}}
\newcommand{\PSL}{\mathrm{PSL}}
\newcommand{\SU}{\mathrm{SU}}
\newcommand{\PGU}{\mathrm{PGU}}
\newcommand{\PSU}{\mathrm{PSU}}
\newcommand{\PSO}{\mathrm{PSO}}
\newcommand{\PGO}{\mathrm{PGO}}
\newcommand{\PCO}{\mathrm{PCO}}
\newcommand{\PGSp}{\mathrm{PGSp}}
\newcommand{\GO}{\mathrm{GO}}
\newcommand{\GU}{\mathrm{GU}}
\newcommand{\GL}{\mathrm{GL}}

\newtheorem{theorem}{Theorem}

\newtheorem{corollary}[theorem]{Corollary}

\newtheorem{definition}[theorem]{Definition}

\newtheorem{lemma}[theorem]{Lemma}

\newtheorem{remark}[theorem]{Remark}

\begin{document}

\title{Criteria for solvable radical membership via $p$-elements}
\author{Simon Guest}
\address{Mathematics\\
University of Southampton\\
Highfield\\
SO17 1BJ\\
United Kingdom}
\email{s.d.guest@soton.ac.uk}
\author{Dan Levy}
\address{The School of Computer Sciences \\
The Academic College of Tel-Aviv-Yaffo \\
2 Rabenu Yeruham St.\\
Tel-Aviv 61083\\
Israel}
\email{danlevy@trendline.co.il}

\maketitle

\begin{abstract}
Guralnick, Kunyavskii, Plotkin and Shalev 
have shown that the solvable radical of a finite group $G$ can be characterized as the set of all $x\in G$ such that $\langle x,y\rangle $ is solvable for all $y\in G$. We prove two generalizations of this result. Firstly, it is enough to check the solvability of $\langle x,y\rangle$ for every $p$-element $y\in G$ for every odd prime $p$. Secondly, if $x$ has odd order, then it is enough to check the solvability of $\langle x,y\rangle $ for every $2$-element $y\in G$.
\end{abstract}

\section{Introduction}

Let $G$ be a finite group 
 and let $R(G)$ denote the solvable radical of $G$; that is, the (unique) largest, solvable, normal subgroup of $G$. We present two characterizations of $R(G)$. The first one is given in Theorem \ref{Th_R(G)Characterization<x,y>yoddu}.

\begin{theorem}
\label{Th_R(G)Characterization<x,y>yoddu}Let $G$ be a finite group and let $ x\in G$. Then $x\in R(G)$ if and only if for all odd primes $p$ and all $p$-elements $y\in G$, the subgroup $\langle x,y\rangle $ is solvable.
\end{theorem}

The next result gives an $R(G)$ membership criterion for odd order $p$-elements. 

\begin{theorem}
 \label{Th_R(G)Characterization<x,y>xoddpy2} Let $G$ be a finite group, and let $x\in G$  be a $p$-element, where $p$ is an odd prime. The element $x$ is contained in $R(G)$ if and only if the subgroup $\langle x,y\rangle $ is solvable for all $2$-elements $y\in G$.
\end{theorem}

Our second characterization of $R(G)$ follows easily from Theorems  \ref{Th_R(G)Characterization<x,y>yoddu} and \ref{Th_R(G)Characterization<x,y>xoddpy2}. First we need some notation.
Denote the order of  $ x\in G$  by $o(x)$ and 
write $o(x) =p_{1}^{\alpha _{1}}\cdots p_{k}^{\alpha _{k}}$, where the $p_i$ are distinct primes.  Then there exists a unique (up to the order of the factors) factorization $x=x_{p_{1}}\cdots x_{p_{k}}$, where  $x_{p_{i}}$ is a $p_i$-element and each $x_{p_i} = x^{k_i}$ for some integer $k_i$.   Note that for any subgroup $A$, the element  $x$ is contained in $A$ if and only if $x_{p_{i}}$ is contained in $A$ for each $i=1, \ldots, k$. Now define $x_{2^{\prime}}=xx_{2}^{-1}(=x_{2}^{-1}x)$ so that $o(x) = o(x_2)o(x_{2'})$. We can now state our second characterization of $R(G)$. 

\begin{corollary}
 \label{Coro_22'R(G)Characterization}Let $G$ be a finite group and let $x\in G $. The element $x=x_2x_{2'}$ is contained in $R(G)$ if and only if both \\
 \emph{(i)} the subgroup $\langle x_{2},y\rangle $ is solvable all $p$-elements $y\in G$ for all odd primes $p$; and \newline
\emph{(ii)} the subgroup $\langle x_{2^{\prime}},y\rangle $ is solvable for all $2$-elements $y\in G$. \newline
In particular, if $x$ has odd order, then $x\in R(G)$ if and only if $\langle x,y\rangle $ is solvable for all $2$-elements $y\in G$.
\end{corollary}

\begin{proof}
If $x\in R(G)$, then (i) and (ii) follow from elementary properties of $R(G)$. Conversely, suppose that (i) and (ii) hold. Now (ii) implies that $\langle x_{p_{i}},y\rangle$ is solvable for all $2$-elements $y\in G$ for each odd $p_{i}$ since $ \langle x_{p_i},y \rangle \leq \langle x_{2^{\prime}},y\rangle$. By Theorem \ref{Th_R(G)Characterization<x,y>xoddpy2},  $x_{p_{i}}$ is contained in  $R(G)$ for each odd $p_{i}$ and thus $x_{2^{\prime}}\in R(G)$.  But $x_{2}\in R(G)$ by (i) and Theorem \ref{Th_R(G)Characterization<x,y>yoddu} and so $x=x_{2}x_{2^{\prime}}\in R(G)$.
\end{proof}

The rest of the paper is organized as follows. In Section \ref{Sect_review} we briefly review some of the previous results that motivate our work. In Section \ref{Sect_NotationAndBackgroundResults}, we describe our notation, and we state various background results on which our proofs rely. In Section \ref{Sect_reduction} we reduce the proofs of Theorems \ref{Th_R(G)Characterization<x,y>yoddu} and \ref{Th_R(G)Characterization<x,y>xoddpy2} to proving certain properties of almost simple groups, and in Section \ref{Sect_ProofsForTheASCase} we prove these properties.

 \section{Elementwise conditions for solvability and solvable radical membership\label{Sect_review}}

 Consider conditions on elements $x \in G$ that ensure either the solvability of $G$ or that $x\in R(G)$. We will call such conditions \emph{elementwise conditions}. Since $G$ is solvable if and only if $G=R(G)$, a criterion for $R(G)$ membership always implies a solvability criterion for $G$. However, the sharpest solvability criteria for $G$ do not arise as special cases of $R(G)$ membership criteria.

 Among the elementwise conditions we can distinguish between \\
  (i) \emph{order arithmetic conditions}, that is, arithmetic conditions based on the orders of the elements of $G$ and $R(G)$; and \\
(ii) \emph{$k$-generator subgroup conditions}, that is, conditions that require the solvability of subgroups generated by $k$ elements with specified properties (for some small integer $k\geq 2$). 

 An example of  an order arithmetic condition is Thompson's solvability criterion (\cite[Corollary 3]{Thompson}, see also \cite{Miller,Hall}), which states that $G$ is solvable if and only if there does not exist a triple $(a,b,c)$ of nontrivial elements of $G$, with pairwise coprime orders, such that $abc=1$. Another suggestion for an order arithmetic condition, whose proof has been reduced to an open question about simple groups, can be found in \cite{KapLev08}.

We review some of the main $k$-generator subgroup conditions, since the results in this paper are of this type.

(a) \emph{Solvability criteria for $G$}

 (a1) Thompson \cite[Corollary 2]{Thompson} (1968): $G$ is solvable if and only if $\langle x,y\rangle$ is solvable for all $x,y \in G$. 

 (a2) Guralnick and Wilson \cite{GW} (2000): $G$ is solvable if and only if the proportion of pairs $(x,y) \in G \times G$ (with $x \ne y$) for which $\langle x,y\rangle $ is solvable is at least $11/30$.

 (a3) The first author \cite{Guest1} (2010) (and (independently) Gordeev, Grunewald, Kunyavskii, and Plotkin in \cite{GGKP} (2010)): $G$ is solvable if and only if $\langle x,x^{y}\rangle $ is solvable for all $x,y\in G$.

 (a4) Kaplan and the second author \cite{KL10} (2010) and both authors \cite{GuLev} (2011) (somewhat improving on \cite{KL10}): $G$ is solvable if and only if for all odd primes $p$, except possibly one, all $p$-elements $x\in G$ and all $2$-elements $y\in G$, the subgroup $\langle x,x^{y}\rangle$ is solvable.

 (a5) Dolfi, Guralnick, Herzog and Praeger in \cite{DGHP} (2011): $G$ is solvable if and only if for every pair $(C,D)$ of distinct conjugacy classes consisting of elements of prime power order, there exists $(x,y) \in C \times D$ such that $\langle x,y\rangle $ is solvable.

(b) \emph{$R(G)$ membership criteria}

 (b1) Guralnick, Kunyavskii, Plotkin and Shalev \cite{GKPS} (2006): 
  an element $x$ is contained in $R(G)$ if and only if $\langle x,y\rangle $ is
solvable for all $y\in G$.

(b2) The first author, Flavell, and Guralnick \cite{FGG} (2010) (and (independently)  Gordeev, Grunewald, Kunyavskii, and Plotkin in \cite{GGKP} (2010)): 
 an element $x$ is contained in $R(G)$ if and only if $ \langle x,x^{g_1},x^{g_2},x^{g_3} \rangle$ is solvable for all $g_1,g_2,g_3 \in G$.
 
 Another perspective from which one can look at these results is, loosely speaking, the complexity of their proofs. Thompson proved his solvability criterion (a1) as a corollary to his classification of the minimal simple groups, while Flavell showed in \cite{Flav95} (1995) that it can be proved using elementary methods. On the other hand, the proofs of (a5) and the $R(G)$ membership criteria in (b) rely on the full classification of finite simple groups. The results of this paper generalize (b1) and are inspired by (a4) (a weaker version of which is obtained as a special case). Our proofs also rely on the classification of finite simple groups. 

\section{Notation and background results\label{Sect_NotationAndBackgroundResults}}

\subsection{Notation} 
\subsubsection*{}
\vspace*{-3mm}

In general, we follow the definitions and notation of \cite{GLS}.
 In particular, $\mathbb{F}_{q}$ will denote a finite field with $q$ elements and characteristic $r$.
 For a group $G$ of Lie type, $\Sigma$ will denote the associated root system (of the untwisted group) and $\Pi$ the corresponding fundamental root system. 
 The Lie rank of $G$
 is equal to $| \Pi |$.
The lowest root relative to $\Pi$ will be denoted by $\alpha _{\ast}$ (and $-\alpha _{\ast}$ denotes the highest root). Let $\widehat{\Pi}$ denote the set of orbits of $\Pi$ under the associated symmetry of the Dynkin diagram, as in \cite[Section 2.3]{GLS}. When $G$ is untwisted, the associated symmetry is trivial and $\Pi = \widehat{\Pi}$.

If $G$ is an almost simple group, then $G_{0}$ will denote its socle (a nonabelian simple group). When $G_0$ is a simple group of Lie type,  $G_{0}^{\ast}=\Inndiag (G_{0})$ denotes the group of inner-diagonal automorphisms of $G_0$. 

 Many cases in our proofs reduce to checking a small number of relatively small groups $G$. We indicate the usage of \textsc{Magma} \cite{MAGMA} in verifying that $G$ has the desired properties by saying that $G$ belongs to an appropriate computerized verification list (CVL for short), which we define in the text.

\subsection{Background Results}

\subsubsection{Results for the reduction step}

The results quoted here are used to reduce the proofs of Theorems \ref{Th_R(G)Characterization<x,y>yoddu} and 
\ref{Th_R(G)Characterization<x,y>xoddpy2} to questions about almost simple groups.

\begin{theorem}{(\cite[Theorem A]{Guest1})}
 \label{Th_A_G1}Let $G$ be a finite group and let $x \in G$ be an element of prime order $p \ge 5$. Then $x\in R(G)$ if and only if $\langle x,x^{g}\rangle$ is solvable for all $g\in G$.
\end{theorem}

\begin{theorem}{(\cite[Theorem A*]{Guest1})}
 \label{ThmA*}Let $G$ be an almost simple group with socle $G_{0}$. Let $x\in G$ have odd prime order $p$. Then one of the following conditions hold. \\ 
 \emph{(1)} There exists $g\in G$ such that $\langle x,x^{g}\rangle $ is not solvable; or \\
\emph{(2)} $p=3$ and either \\
 \indent \emph{(i)} $G_{0} \cong \PSL^{\varepsilon}_d(3)$, $\PSp_d(3)$, $\mathrm{P}\Omega_d^{\varepsilon}(3)$, $E^{\varepsilon}_6(3)$, $E_7(3)$, $E_8(3)$, $F_4(3)$,  or ${^3}D_4(3)$ and $x$ is a long root element in $G_0$, or $G_{0} \cong G_{2}(3)$ and $x$ is a long root or short root element in $G_0$; or  \\
\indent \emph{(ii)} $G_{0}\cong \PSU_{d}(2)$ and $x \in \PGU_d(2)$ has a preimage $x_1 \in \GU_d(2)$ that stabilizes a subspace decomposition $V=V_{1} \perp V_{d-1}$, where $\dim V_i =i$, $V_1$ and $V_{d-1}$ are nondegenerate, $x$ acts trivially on $V_{d-1}$, and $x$ acts on $V_1$ by multiplication by $\lambda \in \mathbb{F}_{4}$ of order $3$.
\end{theorem}


\begin{theorem}{(\cite[Theorem 1.3]{Guest2})}
 \label{Th_xorderp>3Andinvolution} Suppose that $G$ is an almost simple group. If $x\in G$ has prime order $p \ge 5$, then there exists an involution $y\in G$ such that $\langle x,y\rangle $ is nonsolvable.
\end{theorem}

\subsubsection{Results on inner-diagonal automorphisms}

Let $G_0$ be a simple group of Lie type of characteristic $r$ and let $x\in G_0^{\ast}$. Then $x$ is unipotent if it has order a power of $r$ and $x$ is semisimple if it has order prime to $r$. 

\begin{lemma}{(\cite[Lemmas 2.1 and 2.2]{GS})}     
\label{lem:2.2} Let $G_0$ be a simple group of Lie type with $x\in G_0^*$.\\
 \emph{(a)} If $x$ is unipotent, let $P_{1}$ and $P_{2}$ be distinct maximal parabolic subgroups of $G_0$ containing a common Borel subgroup, with unipotent radicals $U_{1}$ and $U_{2}$. Then $x$ is $G_0$-conjugate to an element of $P_{i}\backslash U_{i}$ for $i=1$ or $i=2$. \\
 \emph{(b)} Suppose that $x$ is semisimple and that $x$ is contained in a parabolic subgroup of $G_0^*$. If the rank of $G_0$ is at least 2, then there exists a maximal parabolic $P$ with a Levi complement $J$ such that $x$ is $G_0^*$-conjugate to an element of $J$ not centralized by any Levi component (possibly solvable) of $J$.
\end{lemma}

When $G$ is untwisted, note that Lemma \ref{lem:2.2}(a) can be applied  if there are at least two distinct nodes of the Dynkin diagram. These nodes define a pair of distinct maximal parabolic subgroups. Similarly, when $G$ is twisted, Lemma \ref{lem:2.2}(a) can be applied when there are at least two orbits of nodes under the associated symmetry of the Dynkin diagram. The next lemma is useful for applying Lemma \ref{lem:2.2}(b).

\begin{lemma}\label{Lem_semisimpleInParabolic}
 Let $K$ be a simple group of Lie type of characteristic $r$, and let $x\in K^{\ast}$ have order coprime to $r$. If $O^{r^{\prime}}( C_{K^{\ast}}(x) )$ is nontrivial, then $x$ is contained in a parabolic subgroup of $K^{\ast}$.
\end{lemma}

\begin{proof}
 Assume for a contradiction that $x$ is not contained in any parabolic subgroup of $K^{\ast}$. Let $g\in C_{K^{\ast}}(x)$ be a nontrivial $r$-element. Then, by the Borel--Tits theorem (see \cite[Theorem 3.1.3]{GLS} for example) $ N_{K^{\ast}}(\langle g\rangle )$ is contained in a parabolic subgroup of $K^{\ast}$ and since $x\in N_{K^{\ast}}(\langle g\rangle )$ we have a contradiction.
\end{proof}

\subsubsection{Results on field, graph and graph-field automorphisms}

For $A \in \GL_n(q^{k})$, define the matrix $A^{\varphi _{q}}$ by $(A^{\varphi _{q}}) _{ij}=A_{ij}^{q}$. 

\begin{lemma}{(\cite[Section 7.2]{GL})}
\label{Lem_FieldAutCoset}Let $K$ be a simple group of Lie type and 
 suppose that $\phi \in \Aut ( K)$ is either a field or a graph-field automorphism. Suppose that $\phi' \in \Aut ( K )$ has the same order as $\phi$ and that $\phi K^{\ast} = \phi' K^{\ast}$. Then there exists $y\in K^{\ast}$ such that $\phi ^{\prime}=\phi ^{y}$. In particular, $\phi ^{\prime}$ is a field or a graph-field automorphism.
\end{lemma}

 The conjugacy of graph automorphisms is more complicated. The next lemma combines results from \cite[Lemma 3.7]{moufang} and \cite[Lemma 3.9]{monodromy}, which describe representatives of the $K^{\ast}$-classes of graph involutions when $K=\PSL_{n}^{\varepsilon}(q)$. 
Here $\iota$ denotes the inverse-transpose automorphism of $\PSL_n(q)$.

\begin{lemma}
 \label{Lem_involutiveGraphAutAnepsilon}Let $K= \PSL_{n}^{\varepsilon}(q)$, where $n\geq 3$, and let $v$ be the number of $K^{\ast}$-classes of graph involutions of $K$ (that is, classes of involutions  in $K^{\ast}\Gamma _{K}-K^{\ast}$ if $\varepsilon =1$ and classes of involutions in $K^{\ast}\Phi _{K}-K^{\ast}$ if $\varepsilon =-1$).

 \noindent\emph{(i)} If $n$ is odd, then $v=1$. A representative of the single $K^{\ast}$-class is given by $\iota $ if $\varepsilon =1$, and by $\varphi _{q}$ if $\varepsilon =-1$.

  \noindent \emph{(ii)} If $n$ is even and $q$ is even, then $v=2$. Representatives of the two $ K^{\ast}$-classes are given by $\iota $ and $\iota St$ if $\varepsilon =1$ and by $\varphi _{q}$ and $\varphi _{q}x_{0}(1)$ if $\varepsilon =-1$. The $n\times n$ matrices $S$, $t$ and $x_{0}(1)$ are given by
\begin{equation*}
S=\diag\left[ \left( 
\begin{array}{cc}
0 & -1 \\ 
1 & 0
\end{array}
\right),\ldots,\left( 
\begin{array}{cc}
0 & -1 \\ 
1 & 0
\end{array}
\right) \right], \quad t= \diag\left[ \left( 
\begin{array}{cc}
1 & 1 \\ 
0 & 1
\end{array}
\right),1,\ldots ,1\right],
\end{equation*}
and $x_{0}(1) =I_{n}+E^{(1,n)}$ where $(E^{(1,n)}) _{ij}=\delta _{i1}\delta _{jn}$.

 \noindent \emph{(iii)} If $n$ is even and $q$ is odd, then $v=3$. The centralizers in $K$ of the representatives of the three $K^{\ast}$-classes are subgroups of type $\Sp_{n}(q)$, $O_{n}^{+}(q)$ and $ O_{n}^{-}(q)$ (for both values of $\varepsilon $). If $\varepsilon =1$, then the representatives of the three $K^{\ast}$-classes are given by $\iota S$, $\iota S^{+}$, and $\iota S^{-}$, where
\begin{equation*}
S^{+}=\diag\left[ \left( 
\begin{array}{cc}
0 & 1 \\ 
1 & 0
\end{array}
\right),\ldots,\left( 
\begin{array}{cc}
0 & 1 \\ 
1 & 0
\end{array}
\right) \right], \quad S^{-}=\diag\left[ \left( 
\begin{array}{cc}
0 & 1 \\ 
1 & 0
\end{array}
\right),\ldots,\left( 
\begin{array}{cc}
0 & 1 \\ 
1 & 0
\end{array}
\right),\mu,1\right]
\end{equation*}
 are $n\times n$ matrices  and $-\mu/2 $ is a non-square. For $\varepsilon =-1$, it is shown in \cite[Section 3.11.4]{moufang}, that there exist three choices of  Hermitian form, $f_{1}$, $f_{2}$, $ f_{3}$, such that for  $\GU_{n}(q,f_{i}) =C_{\mathrm{GL}_{n}(q^{2})}(f_{i}\varphi  _{q}\iota ) =X$ and $\PSU_{n}(q,f_{i}) = X^{\prime}/Z(X^{\prime})$, the actions of $\varphi _{q}$  on  $\PSU_{n}(q,f_{i})$ represent all three $K^{\ast}$-classes.
\end{lemma}

\begin{remark} 
 \emph{In  Lemma \ref{Lem_involutiveGraphAutAnepsilon}(ii), if $n=2m$, $\varepsilon =-1$ and $\{  e_{1},\ldots,e_{m},f_{m},\ldots,f_{1}\}$ is an ordered basis of $V$, then the Hermitian form is  defined by $\langle e_{i},e_{j}\rangle =\langle f_{i},f_{j}\rangle =0$ and $\langle e_{i},f_{j}\rangle=\delta _{ij}$.} 
\end{remark}

\begin{lemma}{(\cite[Lemma 3.6]{moufang})}
\label{Lem_involutiveGraphAutE_6}Let $K=E_{6}^{\varepsilon}(q)$ be  simple.

 \noindent \emph{(i)} If $\varepsilon =1$, then there are precisely two $K^{\ast}$-classes of graph  involutions, 
 with representatives $\gamma $, $\gamma x_{-\alpha _{\ast}}(1)$ if $q$ is even and $\gamma $, $\gamma h_{-\alpha _{\ast}}(-1)$ if $q$ is odd, where $\gamma $ is the standard graph automorphism of $E_{6}$.

\noindent \emph{(ii)} If $\varepsilon =-1$, then there are precisely two $K^{\ast}$-classes of graph involutions 
  with representatives $\varphi _{q}$, $\varphi _{q}x_{-\alpha _{\ast}}(1)$ if $q$ is  even and $\varphi _{q}$, $\varphi _{q}h_{-\alpha _{\ast}}(-1)$ if $q$ is odd.
\end{lemma}

\subsubsection{Other background results}

\begin{remark}
 \label{Rem_PrimitivePrimeDivisor}\emph{For a prime power $q$, we recall that $u$ is a primitive prime  divisor  of $q^{e}-1$ if $u$ divides $q^{e}-1$ but does not divide $q^{i}-1$  for $i=1,\ldots, e-1$. A primitive prime divisor of $q^{e}-1$ exists if $e\geq 3$ and $(q,e) \neq (2,6)$ (see \cite[Theorem 5.2.14]{KL} for example). Moreover, since $q^{e} \equiv 1 \pmod{u}$, we have $u -1= ke$ for some $k \ge 1$ by Fermat's little theorem.   In particular $ u\geq e+1$.}
\end{remark}

\begin{lemma}
 \label{Lem_CentralizerOfFieldAutOfPSL_2(q)} Let $K=\PSL_{2}(q)$ and suppose that $q=q_{0}^{k}$ 
 where $k$ is prime. Let $x$ be the automorphism of $\PSL_{2}(q)$  induced by the field automorphism $\varphi _{q_{0}}$ of $\GL_{2}(q)$. 

 \noindent \emph{(i)} If $k=2$ and $q$ is odd, then $C_{K}(x) \cong \PGL_{2}(q_{0})$.

 \noindent \emph{(ii)} Otherwise (that is, if $k$ is odd or if $q$ is even) we have $C_{K}(x) \cong \PSL_{2}(q_{0})$.
\end{lemma}

\begin{proof}
 Note that $o(x) =k$ and since $\PGL_{2}(q) \cong \mathrm{  Inndiag}( \PSL_{2}(q) )$ it follows that 
 \[C_{\PGL_{2}(q)}(x)\cong \mathrm{Inndiag}( O^{r^{\prime}}(C_{\PSL_{2}(q)}(x) )) \cong  \PGL_{2}(q_{0})\]
  (see \cite[Proposition 4.9.1]{GLS} for example).
  Now $C_{K}(x)=K\cap  C_{\PGL_{2}(q)}(x)$ and we use  \cite[Proposition 4.5.3]{KL} in order to obtain the precise structure of $K\cap C_{\PGL_{2}(q)}(x)$.
\end{proof}

\begin{lemma}
 \label{Lem_CentralizerCosetArgument}Let $G$ be a finite group, let $x\in  \Aut (G)$ have prime power order $p^{\alpha}$, and let $M=C_{G}(x)$. Suppose that $p$ divides   $| G:M|$, and that either $M = N_G(M)$  or $Z(M) =1$. Then there exists a conjugate of $M$ that is normalized but not centralized by $x$.
\end{lemma}

\begin{proof}
 By our assumptions, $p$ divides the number of fixed points of the action of $\langle x\rangle $ on the set of right cosets of $M$ in $G$ via $  (Mg) ^{x}=Mg^{x}$. Since $M$ is a fixed  point for this action  
   there exists $g\in G-M$  such that $(Mg) ^{x}=Mg^{x}=Mg$. It follows that $g^{x}=m_{0}g$  for some $m_{0}\in M$. We have 
\begin{equation*}
(M^{g}) ^{x}=(M^{x}) ^{g^{x}}=(M)
^{g^{x}}=M^{m_{0}g}=M^{g}\text{,}
\end{equation*}
 and thus $x$ normalizes $M^{g}$. Assume for a contradiction that $x$ centralizes $  M^{g}$. Then $M^{g}\leq M$ implying that $M^{g}=M$. Thus we may assume that $M$ is not  self-normalizing in $G$ ($g\in G-M$), and therefore we may assume that $Z(M) =1$.  Hence, for every $m\in M$, we have $m^{g}=m^{g^{x}}$, and therefore $  g^{x}g^{-1}=m_{0}\in Z(M)$ and $m_0=1$. But $g^{x}\neq g$ since $g\in G-M$, which is a contradiction.
\end{proof}

\section{Reduction to the almost simple case \label{Sect_reduction}}

 The first step in proving Theorems \ref{Th_R(G)Characterization<x,y>yoddu}  and \ref{Th_R(G)Characterization<x,y>xoddpy2} is to reduce the proof to a question about almost simple groups. In the  following, $F(G)$ denotes the Fitting subgroup of $G$, $F^{\ast}(G) =F(G) E(G)$ denotes the generalized  Fitting subgroup of $G$, where $E(G)$ is the layer, and $  soc(G)$ is the product of all minimal normal subgroups of $G$.  Note that if $R(G) =1$, then $F(G) =1$ and $F^{\ast}(G) =E(G) =soc(G)$ where, in this  case, $soc(G)$ is a direct product of simple,  nonabelian groups. Each of  these simple groups is  referred to as a \emph{component} of $G$. Furthermore, since $F^{\ast}(G) =soc(G)$ has a trivial centre, and since $C_{G}( F^{\ast}(G) ) \leq F^{\ast}(G)$ (see \cite[(31.13)]{Asbook} for example), we have $  C_{G}( F^{\ast}(G) ) =1$, and hence $G$ acts on $  F^{\ast}(G)$ by conjugation as a group of automorphisms embedded in $\Aut ( F^{\ast}(G) )$. 

\subsection{Reduction of Theorem \protect\ref{Th_R(G)Characterization<x,y>yoddu} to the almost simple case}

\begin{lemma}
 \label{Lem_xCharacterizationNotInN_G(L)} Let $G$ be a finite group such that $F(G)=1$. Let $x\in G$ and let $L$ be a component of $G$ such that $ x\notin N_{G}(L)$. Then there exist an odd prime $p$ and a $p$-element $y\in G$ such that $\langle x,y\rangle $ is not solvable.
\end{lemma}

\begin{proof}
Suppose that the claim is false and let $G$ be a minimal counterexample.  
Consider $\{L^{x^{i}}\mid i=0,1 \ldots\}$, which is the orbit of $L$ under the conjugation  action of $\langle x\rangle $. For each $z\in \langle  x\rangle $, we have either $L\cap L^{z}=1$ and $[  L,L^{z}] =1$, or $L^z=L$. Let $t=|  \{ L^{x^{i}}\mid i=0,1 \ldots\} |$ and note that $t\geq 2$ since $  x\notin N_{G}(L)$. 
Consider $H= \langle x,L \rangle$ and note that  $x \not\in R(H)$. Indeed, if $x \in R(H)$, then on the one hand, $[x,l] \in R(H)$ for all $l \in L$. On the other hand,   we have $[x,l] \in L \times L^x$, thus $[x,l] \in R(H) \cap (L \times L^x) =\{1\}$ for all $l \in L$, which contradicts our assumption that $x \not\in N_G(L)$.  Since $x \not\in R(H)$, the group $H/R(H)$ satisfies the assumptions of the lemma, so, by minimality of $G$, we may assume that $G=H=\langle x,L \rangle$ and $R(G)=1$. 

We have $F^*(G) =soc(G) = L_0 \times L_1 \times	\cdots L_{t-1}$ where each $L_i = L^{x^{j}}$ for some integer $j$ and we may assume that $L_0=L$.
  Now $\Aut ( F^{\ast}(G) ) \cong ( \mathrm{Aut}(L) )^{t}\rtimes S_{t}$ and since $F(G) =1$, we can identify $G$ with a  subgroup of $\Aut ( F^{\ast}(G))$ and write $x=(\sigma _{1},\ldots,\sigma _{t}) \tau $ where $\sigma _{i}\in \Aut (L_{i-1})$ and $\tau \in S_{t}$ is  a $t$-cycle. By relabelling the $L_{i}$ if necessary we may assume that $\tau =(1,2,\ldots,t)$, and $\tau ^{-1}(\sigma  _{1},\ldots,\sigma _{t}) \tau =(\sigma _{t},\sigma _{1},\ldots,\sigma  _{t-1})$. By Theorem \ref{ThmA*}, there exist   $l_{1},l_{2}\in L$ of prime order $p \ge 5$ such that $\langle l_{1},l_{2}\rangle $ is not solvable. Let $y=(l_{1},1,\ldots,1,l_{2}^{\sigma _{t}^{-1}}) \in G$. Then $y$ has order $p$ and
\begin{align*}
x^{-1}yx& =\tau ^{-1}(\sigma _{1}^{-1},\ldots,\sigma _{t}^{-1})
(l_{1},1,\ldots,1,l_{2}^{\sigma _{t}^{-1}}) (\sigma
_{1},\ldots,\sigma _{t}) \tau = \\
& =\tau ^{-1}(l_{1}^{\sigma _{1}},1,\ldots,1,l_{2}^{\sigma
_{t}^{-1}\sigma _{t}}) \tau =(l_{2},\ldots,) \text{.}
\end{align*}
 Now let $\pi  _{1}:G\rightarrow L_{0}$ be the projection homomorphism onto $L_0$. Then $\pi _{1}(\langle y^{x},y\rangle )$ contains a copy of $\langle l_{1},l_{2}\rangle $, hence $\langle x,y\rangle $ is not solvable.
\end{proof}

\begin{lemma}
\label{Lem_ReductyodduCharacterizationToAlmost}
 If $G$ is a minimal counterexample to Theorem \ref{Th_R(G)Characterization<x,y>yoddu},  then $G$ is almost simple. In particular,   if Theorem \ref{Th_R(G)Characterization<x,y>yoddu} holds for almost simple groups, then it holds for all finite groups.
\end{lemma}

\begin{proof}
 Let $G$ be a minimal counterexample to Theorem \ref{Th_R(G)Characterization<x,y>yoddu}. We will show that $G$ is an almost simple group.  By assumption there exists $x\in G-R(G)$ such that for every odd prime $p$, $\langle x,y\rangle $ is solvable for all $p$-elements $y\in G$.

 Suppose that $R(G)$ is nontrivial. Set $\overline{G}=G/R(G)$ and  note that $| \overline{G}| <| G|$. On  the one hand, $\overline{x}=xR(G) \notin R(\overline{G})$  (which is trivial) because $x\notin R(G)$. On  the other hand, $\langle \overline{x},\overline{y}\rangle $ is  solvable for all $p$-elements $\overline{y}\in \overline{G}$, for all odd primes $p$. This implies, by minimality of $G$, that $\overline{x}\in R(\overline{G})$. Thus we have a contradiction.

 So we may assume that $R(G) =F(G) =1$. Since $x\notin R(G)$ we know that $G$ is nonsolvable and $G$ has at least one component $L$.  If $x\notin  N_{G}(L)$, then we can apply Lemma \ref{Lem_xCharacterizationNotInN_G(L)}. Therefore we may assume that $x\in N_{G}(L)$ for each component $L$ of $G$. Since $ C_{G}( F^{\ast}(G) ) =1$, 
 there exists a component $L$ of $G$ on which $\langle x\rangle $ acts  nontrivially. Let $\langle  x^{d}\rangle$ be the kernel of this action. If $x^{d}\neq  1$, then $(G,x^d)$ is also a minimal counterexample and so we may assume without loss of generality that $x^d=1$ and therefore that $\langle x\rangle $ acts faithfully on $L$.
In particular, $\langle x,L\rangle $ embeds in $\Aut (L)$ and since $G$ is a minimal counterexample, we have $G= \langle x,L \rangle$. Therefore $L \unlhd G \le \Aut(L)$ and $G$ is almost simple. 
\end{proof}

\begin{remark}
 \label{Rem_CanAssumexTobep-ele}\emph{Lemma   \ref{Lem_ReductyodduCharacterizationToAlmost} shows that in order to prove Theorem \ref{Th_R(G)Characterization<x,y>yoddu}, it suffices to prove that for all almost simple groups $G$ and for all nontrivial $x \in G$, there exist an odd  prime $p$ and a $p$-element $y\in G$ such that $\langle  x,y\rangle $ is nonsolvable. In fact, since $\langle x^{k},y\rangle$ being nonsolvable implies that $\langle x,y\rangle$ is nonsolvable, it suffices to check all $x \in G$ of prime order.} 
\end{remark}

\begin{lemma}
 \label{Lem_InAlmostForxp>3ExistsOdduySt<x,y>NS}Let $G$ be an almost simple  group and let $x\in G$ of prime order $u \ge 5$. Then there exists an odd prime $p$ and a $  p$-element $y\in G$ such that $\langle x,y\rangle $ is not solvable.
\end{lemma}

\begin{proof}
 By Theorem \ref{ThmA*} there exists $g\in G$ such that $\langle  x,x^{g}\rangle $ is not solvable. Hence, taking $p=u$ and $y=x^{g}$, $\langle x,y\rangle $ is not solvable.
\end{proof}

 Combining Lemma \ref{Lem_ReductyodduCharacterizationToAlmost}, Remark   \ref{Rem_CanAssumexTobep-ele} and Lemma \ref{Lem_InAlmostForxp>3ExistsOdduySt<x,y>NS}, it is clear that Theorem \ref{Th_R(G)Characterization<x,y>yoddu} will follow from the next theorem.

\begin{theorem}
 \label{Th_InAlmostForxp=2,3ExistsOdduySt<x,y>NS}Let $G$ be an almost simple  group. If $x\in G$ has order $2$ or $3$, then there exist an odd prime $p$ and a $p$-element $y$ for which $\langle x,y\rangle $ is not solvable.
\end{theorem}

\subsection{Reduction of Theorem \protect\ref{Th_R(G)Characterization<x,y>xoddpy2} to the almost simple case}

 Let $G$ be a minimal counterexample to Theorem \ref{Th_R(G)Characterization<x,y>xoddpy2}, so that $G$ contains a $p$-element $x$ ($p$ odd) such that $x \not \in R(G)$ and $\langle x,y\rangle $ is solvable for all $2$-elements $y$. Since $G$ is a minimal counterexample, $R(G)$ must be trivial. We therefore may as well assume that $x$ has order $p$.  
 
 Let $L$ be a component of $G$.   First suppose that $x \not\in N_G(L)$. Then as in the proof of Lemma \ref{Lem_xCharacterizationNotInN_G(L)}, we may assume that  $G=\langle x,L\rangle =L^{p}\langle x\rangle \leq \Aut (L)\rtimes S_{p}$. So $  x=(\sigma _{1},\ldots,\sigma _{p})\tau$, where $\sigma _{i}\in \Aut (L)$ and $\tau \in S_{p}$ is a $p$-cycle. Without loss of generality, we  may assume that $\tau =(1,2,\ldots,p)$ and moreover, conjugating by a suitable $  (u_{1},\ldots ,u_{p})\in \Aut (L)^{p}$ we may assume that $  x=(\sigma ,1,\ldots ,1)(1,2,\ldots,p)$. Furthermore, since $x^{p}=(\sigma,\sigma ,\cdots ,\sigma )=1$, we have $\sigma=1$ and $x=(1,2,\ldots,p)$. Now, by \cite[Theorem A]{MSW}, there exist three involutions in $L$ that generate $L$ unless $  L=\PSU_{3}(3)$, in which case it is easily verified that there exist three $2$-elements in $L$ that generate $L$. So let $y=(y_{1},1,\ldots,1,y_{2},y_{3})\in  L^{p}\leq G$ where the $y_{i}$ are $2$-elements such that $L=\langle  y_{1},y_{2},y_{3}\rangle $. Then $y$ is a $2$-element of $G$ and $  y=(y_{1},1,\ldots,1,y_{2},y_{3})$, $y^{x}=(y_{3},y_{1},1,\ldots,1,y_{2})$ and $  y^{x^{2}}=(y_{2},y_{3},y_{1},1,\ldots,1)$ are all contained in $\langle x,y\rangle $. These three elements  generate a subgroup of $\langle x,y\rangle $ whose projection onto  the first component of $L^{p}$ contains $\langle  y_{1},y_{2},y_{3}\rangle =L$, which is not solvable. Thus, assuming that $x \not\in N_G(L)$ leads to a contradiction, so we conclude that $x \in N_G(L)$.  Now arguing as in the proof of  Lemma \ref{Lem_ReductyodduCharacterizationToAlmost}  shows that  $G$ is almost simple.  Moreover, Theorem \ref{Th_xorderp>3Andinvolution} implies that $p=3$ and it therefore remains to prove the following. 
\begin{theorem}
 \label{Th_3plus2}Let $G$ be an almost simple group and let $x\in G$ of order $3$. Then there exists a $2$-element $y\in G$ such that $\langle x,y\rangle $ is not solvable. 
\end{theorem}

\section{Proofs for the almost simple case \label{Sect_ProofsForTheASCase}}

\subsection{Proof of Theorem \protect\ref{Th_InAlmostForxp=2,3ExistsOdduySt<x,y>NS}}

\subsubsection{The case $o(x) =3$ of Theorem \protect\ref{Th_InAlmostForxp=2,3ExistsOdduySt<x,y>NS} \label{Sect_ProofFirstTho(x)=3}}

 In this section, we prove that if $G$ is almost simple and $x \in G$  has order $3$, then there exist an odd prime $p$ and a $p$-element $y \in G$ such that $\langle x,y \rangle$ is not solvable.
 
\begin{definition}
  The list $CVL1$ consists of the following groups: $G_{2}(3)$, $\PSL_{3}(3)$, $\PSp_{4}(3)$ and $\PSU_{3}(3)$. For every $G_0 \in  CVL1$ we have verified, using \textsc{Magma}, that if $x\in \Aut (G_0)$ has order $3$, then there exist an odd prime $p$ and a $p$-element $y\in G_0$ for which $\langle x,y\rangle$ is not solvable.
\end{definition}

\begin{proof}[Proof of Theorem \protect\ref{Th_InAlmostForxp=2,3ExistsOdduySt<x,y>NS} for $o(x) =3$]
 By Theorem \ref{ThmA*}, we can either take $p=3$ and $y$ some conjugate of $  x $, or we may assume that Theorem \ref{ThmA*}(2) holds. In the latter case, we show that we can reduce to the case of $\PSL_3(3)$, $\PSp_4(3)$ or $\PSU_3(3)$, all of which are contained in $CVL1$. We split the discussion into the following cases: 
 
 1. The socle $G_0$ is one of the exceptions listed in Theorem \ref{ThmA*}(2)(i), and $x$ is a  long root element in $G_0$.  Since $G_{2}(3) \in CVL1$ we may  assume that $G\neq G_{2}(3)$.  By minimality we have $G=G_0$. The Dynkin diagram of $G$ has $n\geq 2$ nodes. In all cases the  long root elements of $G$ form a single $G_{0}^{\ast}$-conjugacy class of $G_{0}^{\ast}$ (see \cite[Example 3.2.6]{GLS}), and thus we may assume $x=x_{-\alpha_{\ast}}(1)$. Now we use subsystem subgroups (see \cite[Section 2.6]{GLS} for example) to find a  suitable subgroup $A$, which will yield one of the groups in $CVL1$. Let $s$ be a fundamental root adjacent to $\alpha _{\ast}$ in the extended Dynkin diagram. We define a  subsystem $\Sigma _{0}$ of $\Sigma $ with respect to $G$ (\cite[Definition 2.6.1]{GLS}). There are two possibilities. 
 
 a. $G\neq {^2}A_{n}(3)$. Choose $\Sigma_{0}$ to be the intersection of the $\mathbb{Z}$-span of $\{ \alpha _{\ast},s\}$ and $\Sigma $. Then $\Sigma_{0}$ is of type $A_{2}$ or $C_{2}$, and taking $  w=1$ in  \cite[Definition 2.6.1]{GLS}, it follows that that $A=\langle x_{\alpha  _{\ast}}(1),x_{-\alpha_{\ast}}(1),x_{s}(1),x_{-s}(1)\rangle $ is a subgroup of $G$ isomorphic to either $A_{2}(3)$ or $C_{2}(3)$.

  b. $G={^{2}}A_{n}(3)$. For $n=2$ take $A=G$.
  For $n \ge 3$, let 
  \[A=\langle x_{\alpha _{\ast}}(1),x_{-\alpha_{\ast}}(1),x_{-\alpha_{1}}(1)x_{-\alpha _{n}}(1),x_{\alpha _{1}}(1)x_{\alpha _{n}}(1)\rangle\]
    and note that  $A\cong C_{2}(3)$ contains $x$.

 Thus, in either case, the image of $x$ in $A/Z(A)$ is nontrivial and $A/Z(A)$ is one of $\PSL_{3}(3)$, $\PSp_{4}(3)$, $\PSU_{3}(3)$; these groups are contained in $CVL1$.

 2. $G=\PGU_{d}(2)$, where $d\geq 4$, and $x$ is the  image of  $x_{1}\in \GU_{d}(2)$ as described in Theorem  \ref{ThmA*}(2)(ii).  
 Clearly $x_{1}$ stabilizes a subspace  decomposition $V=V_{4}\perp V_{d-4}$, with $x_1$ acting noncentrally on a nondegenerate $4$-dimensional subspace $V_{4}$. So $x_{1}\in \GU_{4}(2)$. Therefore we can reduce to the case of $G_0=\PSU_4(2)$, and  $\PSU_{4}(2)\cong \PSp_{4}(3)$ is contained in $ CVL1$.
\end{proof}

\subsubsection{The proof of  Theorem \protect\ref{Th_InAlmostForxp=2,3ExistsOdduySt<x,y>NS}  for involutions\label{Sect_ProofFirstTho(x)=2}}

 In this section, $(G,x)$ is a minimal counterexample to the claim of  Theorem \ref{Th_InAlmostForxp=2,3ExistsOdduySt<x,y>NS} when $x$ is an involution. Thus $G$ is almost simple and  $\langle x,y\rangle $ is solvable for every $p$-element $y\in G$ for every odd prime $p$. 
 
\begin{definition}
 The list $CVL2$ consists of the following groups: $A_{6}$, $\PSL_{3}(2)$, $\PSU_{4}(2)$, $\PSU_{5}(2)$, $^{3}D_{4}(2)$, $\PSL_{3}(3)$, $\PSL_{4}(3)$, $\mathrm{P}\Omega _{7}(3)$, $\PSp_{4}(3)$, $\PSp_{6}(3)$, $G_{2}(3)$, $\PSU_{4}(3)$, $^{2}D_{4}(3)$, $\PSU_{3}(3)$, $\mathrm{P}\Omega _{8}^{\pm}(2)$, $\mathrm{P}\Omega _{8}^{\pm}(3)$, $F_{4}(2)$, $^{2}F_{4}(2) ^{\prime}$. For each $G_0\in CVL2$, we have verified using \textsc{Magma} that if  $x \in \Aut (G_0)$ has order $2$,  then  there exist an odd prime $p$ and  a $p$-element $y\in G_0$ for which $\langle x,y\rangle $ is not solvable. \end{definition}


\begin{lemma}
 \label{Lem_SingleAut(G)classOf involutions} If $(G,x)$ is a minimal counterexample to Theorem \ref{Th_InAlmostForxp=2,3ExistsOdduySt<x,y>NS}, then $G$ has more than one $\Aut (G_0)$-conjugacy class of involutions.
\end{lemma}
\begin{proof}
  Suppose that $G$ has only one $\Aut(G_0)$-conjugacy class of involutions. Pick $z \in G_0$ of prime order $ p\ge 5$. By Theorem \ref{Th_xorderp>3Andinvolution}, there exists an involution $y \in G_0$ such that $\langle z,y\rangle $ is nonsolvable. But $y^{w}=x$ for some $w \in  \Aut(G_0)$ and so $\langle x, z^{w^{-1}} \rangle \cong \langle y,z \rangle$ is nonsolvable. 
\end{proof}

\begin{proof}[Proof of Theorem \protect\ref{Th_InAlmostForxp=2,3ExistsOdduySt<x,y>NS} for $x$ an involution]
We split the discussion according to the isomorphism type of $G_{0}$.

 (A) $G_{0}\cong A_{n}$, $n\geq 5$. Since all odd $p$-elements of $\Aut (A_{n})$ are in $A_{n}$ it is sufficient to prove  the claim for $G=\Aut (A_{n})$. Since $A_{6}\in CVL2$  we have $n\neq 6$, and hence $G=S_{n}$. Considering representatives of  distinct conjugacy classes of involutions in $G$, we may assume that $x=(1,2)$, $(1,2)(3,4)$ or $(1,2)(3,4)(5,6)\sigma $ where $\sigma \in \mathrm{Sym}\{7,\ldots ,n\}$  satisfies $\sigma ^{2}=1$. Let $y=(1,2,3,4,5)\in G$. In all cases, $\langle x,y\rangle $ contains $\mathrm{Alt}\{1,\ldots ,5\}$, which is not solvable, and we have a contradiction.

(B) \emph{Simple groups of Lie type of characteristic $r$}

(B.1) $x \in G_{0}^{\ast} = \Inndiag(G_0)$. We split  the analysis according to whether $x$ is unipotent ($r=2$) or  semisimple ($r>2$).

 (B.1.1) \emph{Unipotent involutions ($r=2$).} We have $x\in G_{0}$ since $| \mathrm{Outdiag}  (G_{0})|$ is not divisible by $r$ \cite[Theorem 2.5.12(c)]{GLS}, and so $G=G_{0}$ by minimality.  Suppose that $|\hat{\Pi}| >1$. By Lemma   \ref{lem:2.2} and the discussion following it, we may assume that $x\in  P\backslash U$ where $P$ is a standard maximal parabolic subgroup of $G_0$ and $U$ is its nontrivial unipotent radical.  Since $x\in P\backslash U$,  the image $\overline{x}$ of $x$ in $P/U\cong L=MH$ is nontrivial, $H$ is the Cartan subgroup and $M$ is a central product of  groups of Lie type of characteristic $2$ corresponding to the Dynkin diagram of $G_0$ with one node deleted (see  \cite[Theorem 2.6.5]{GLS}).  Since $x$ has order $r=2$ we have $\overline{x} \in M$  and one of the groups $M_{0}$ in the central product is normalized but not centralized by $\overline{x}$. Set $A=\langle \overline{x}\rangle M_{0}$. We note that $|  A| <| G_{0}|$ since $A\leq L \lneqq G_{0}$. If  all components of $M$ are nonsolvable, then so is $M_{0}$, hence $A/Z(A)$ is almost simple, and the image $x_{1}$ of $\overline{x}$ in $A/Z(A)$  is nontrivial. Thus $(G,x)$ cannot be a minimal counterexample. 


 The groups $G_{0}$ that are not eliminated by the last argument satisfy either    (i) $| \hat{\Pi}| =1$, or (ii) for every pair of distinct  maximal parabolic subgroups of $G_{0}$ containing a common Borel subgroup,  at least one of the two parabolics has a Levi complement with a  solvable component. In case (ii), the  solvable component is of type $A_{1}(2)$, ${^{2}}A_{2}(2)$ or $^{2}B_{2}(2)$ (see \cite[Theorem 2.2.7]{GLS}).   It is straightforward, using   \cite[Proposition 2.6.2 and Theorem 2.6.5]{GLS} for example, to obtain the list of $G_0$ satisfying (ii).   The list of groups $G_{0}$ satisfying (i) or (ii) consists of $\PSL_{2}(2^{a})$  ($a\geq 2$), $\PSU_{3}(2^{a})$   ($a\geq 2$), $^{2}B_{2}(2^{a})$ ($a \ge 3$  odd), $\PSL_{3}(2)$, $\PSU_{4}(2)$, $\PSU_{5}(2)$ and $ {^3}D_{4}(2)$. We consider the groups on this list.

 We note that $\PSL_{2}(2^{a})$ and $\PSU_{3}(2^{a})$  have a single class of involutions (transvections) (see \cite[p. 103]{GLS} for example). The same is true for ${^2}B_{2}(2^{a})$ by  \cite[Proposition 8]{2B2}.  So $(G,x)$ cannot be a minimal counterexample in all these cases by Lemma \ref{Lem_SingleAut(G)classOf involutions}. The remaining groups $\PSL_{3}(2)$, $\PSU_{4}(2)$, $\PSU_{5}(2)$, ${^3}D_{4}(2)$ all belong to $CVL2$.

 (B.1.2) \emph{Semisimple involutions.} Suppose that $r>2$, and so $x$ is semisimple and suppose that $|\hat{\Pi} | >1$. Since $O^{r^{\prime}}( C_{G^{\ast}}(x))$ is nontrivial  
 (see \cite[Table 4.5.1]{GLS}),
   $x$ belongs to a maximal parabolic subgroup of $G_0^*$ by Lemma \ref{Lem_semisimpleInParabolic}.  By Lemma \ref{lem:2.2}(b), $x$ is $G_0^*$-conjugate to an element $z$ that acts noncentrally on each component of the Levi complement of some  maximal parabolic subgroup $P^*$ of $G_{0}^*$. Without loss, we may assume that $x=z$ and therefore $x$ acts non-centrally on each component of the Levi complement of the maximal parabolic $P^* \cap G$ of $G$.  If one of the components $M_{0}$ is nonsolvable, then we can find $A<G$ and $x_{1}$  such that $( A/Z( A) ,x_{1})$ contradicts the  minimality of $(G,x)$ as in (B.1.1).  The list of groups $G_{0}$ with $|\hat{  \Pi}| >1$ and a maximal parabolic subgroup whose components  are all solvable or with $|\hat{\Pi}| =1$  consists of $\PSL_{2}(q)$ ($q\geq 5$ odd), $\PSU_{3}(q)$ ($q\geq 3$ odd), ${^{2}}G_{2}(3^{a})$, $^{3}D_{4}(3)$, $\PSL_{3}(3)$, $\PSL_{4}(3)$, $\mathrm{P}\Omega _{7}(3)$, $\PSp_{4}( 3)$, $\PSp_{6}(3)$, $\mathrm{P}\Omega _{8}^{+}(3)$, $G_{2}( 3)$, $\PSU_{4}(3)$, $\mathrm{P}\Omega _{8}^{-}(3)$.  
   
 (B.1.2.1) $G_{0}\cong \PSL_{2}(q)$, $q=r^{a}\geq 5$ odd. Recall  that $| G_{0}| =q(q^{2}-1)/2$ and we have $G=\langle  G_{0},x\rangle = \PSL_2(q)$ or $\PGL_2(q)$. By \cite[Table 4.5.1]{GLS}, we have $|C_{G_{0}^{\ast}}(x)| \leq 2(q+1)$ and by   \cite[Theorem 4.2.2(j)]{GLS}, we have $x^{G} =x^{G_0^*}$ and $|x^{G}| \geq | G_{0}^{\ast}|/2(q+1)=q(q-1)/2$. Let $y\in G_{0}$ be a $p$-element where $p$ is an odd prime to be specified, and $M$ a maximal subgroup of $G$ containing $y$. Using the lists of maximal subgroups of $\PSL_2(q)$ and $\PGL_2(q)$  (see \cite[Theorems 2.2 (Dickson) and 3.5]{Giu07}), we have two cases to consider:

 (B.1.2.1.1) $G_{0}\cong \PSL_{2}(r)$ (that is, $a=1$ and $q=r$). For $q=5$ we have $  G_{0}=\PSL_{2}(5)\cong A_{5}$, which has been eliminated in (A). Hence we may assume that $q \geq 7$. Choose $p=r$. Then $M$ must be a Frobenius group of order $q(q-1)/2$ for $G= \PSL_2(r)$ and of order $q(q-1)$ for $G=\PGL_2(r)$, where $\langle y\rangle $ is the normal Frobenius kernel of order $q$. Any involution in $M$ must be the  unique involution in one of the cyclic Frobenius complements, and so $M$ has $q$ involutions. Since $|x^G| \geq q(q-1)/2$, there is a conjugate $x_{1}$ of $x$ that does not normalize $M$. Moreover, let $g\in G$ be such that $M^{g}\neq M$. We claim that $y\notin M^{g}$. Note that $M=N_{G}(\langle y\rangle)$ (otherwise $\langle y\rangle \trianglelefteq G$  hence $\langle y\rangle \trianglelefteq G_{0}$, which is a contradiction since $G_0$ is simple). If $y\in M^{g}$, then $\langle y\rangle $ coincides with the unique normal subgroup of $M^{g}$ of order $q$ and hence $  M^{g}=N_{G}(\langle y\rangle ) =M$, which is a  contradiction. Thus $y^{x_{1}}\notin M$, and $\langle  y,y^{x_{1}}\rangle \leq G_{0}$ is not contained in $M$ or any of its  conjugates. It follows that $\langle y,y^{x_{1}}\rangle =G_{0}$.  This implies that $\langle x_{1},y\rangle $ is nonsolvable, which is a contradiction.

 (B.1.2.1.2) $G_{0}\cong \PSL_{2}(r^a)$, $a\geq 2$. Choose $p$ to be a primitive prime divisor of $r^{2a}-1$ (see Remark \ref{Rem_PrimitivePrimeDivisor}). Then $p\geq 5$, and $p$ divides $q+1$. Hence $M$ is either a dihedral group of order $q+1$ if $  G=\PSL_2(q)$ or order $2(q+1)$ if $G=\PGL_2(q)$ or an $A_{5}$ (if $p=5$). If $M$ is  dihedral, then $y$ belongs to the unique cyclic subgroup $C \le M$ of index $2$, and $M=N_{G}(C)  =N_{G}(\langle y\rangle)$. Therefore, if $M_{1}\neq M$ is also a maximal dihedral group of $G$ of the same order, then  it does not contain $y$. Also, since $| x^G| \geq q(q-1)/2$ and $q\geq 9$ we have $|x^G| \geq  | M|$ and there is a conjugate $x_{1}$ of $x$ such that $\langle y,y^{x_{1}}\rangle$ is not  contained in $M$ or any of its conjugates. Suppose that $\langle  y,y^{x_{1}}\rangle $ is contained in a maximal subgroup of $G_{0}$ that is isomorphic to $A_{5}$. Then either $\langle  y,y^{x_{1}}\rangle =\langle y \rangle $,  contradicting the fact that $\langle y,y^{x_{1}}\rangle \nleqslant M$, or $\langle  y,y^{x_{1}}\rangle \cong A_{5}$, which is nonsolvable. Otherwise $\langle y,y^{x_{1}}\rangle =G_{0}$. Thus, in all cases, $\langle x_{1},y\rangle $ is nonsolvable, which is a contradiction.

 (B.1.2.2) $G_{0}\cong  {^2}G_{2}(3^{a})$, ${^3}D_{4}(3)$ or $\PSU_{3}(q)$ ($q\geq 3$ odd).  Here there is a unique class of involutions (see  \cite[Table 4.5.1]{GLS} for example) and Lemma \ref{Lem_SingleAut(G)classOf involutions} eliminates these cases.

 (B.1.2.3) The remaining possibilities for $G_0$ are $\PSL_{3}(3)$, $\PSL_{4}(3)$, $\mathrm{P}\Omega _{7}(3)$, $\PSp_{4}(3)$, $\PSp_{6}(3)$, $\mathrm{P}\Omega _{8}^{+}(3)$, $G_{2}(3)$, $\PSU_{4}(3)$, or $\mathrm{P}\Omega _{8}^{-}(3)$. These groups are contained in $CVL2$.

\noindent (B.2)  \emph{Field involutions.} \\
  Suppose that $x$ belongs to a $G_{0}^{\ast}$-coset of $\Aut(G_0)$  represented by a field automorphism. By  Lemma \ref{Lem_FieldAutCoset} we may assume that $x$ is a standard field  automorphism of $G_{0}$. Since $x$ is induced by an order $2$ automorphism  of $\mathbb{F}_{r^a}$, $a$ must be even and $r^a\geq 4$. 
  
  Note that $G_{0}$ is not a Suzuki--Ree group, since if it were, then $\mathrm{Out}(G_0)$ would have odd  order. If $G_{0}$ is a Steinberg group $^{d}\Sigma (q)$, then $d$ and $o(x) =2$ must be coprime, thus leaving $G_{0}\cong  {^3}D_{4}(q)$ as the only possibility. Now ${^3}D_{4}(q)$ and all the simple untwisted groups of Lie type, except $  A_{1}(q)$, have a proper subgroup $A\cong \SL_{2}(q)$ generated by root subgroups of $G_0$ (\cite[Theorem 3.2.8]{GLS}). So we may assume that $A$ is normalized by $x$, and $C_{A}(x) \cong \SL_{2}(q^{1/2})$, and in particular, $x$ does not centralize $A$. Hence, by minimality  of $(G,x)$, we have $G_{0}\cong \PSL_{2}(q)$. Let $M=C_{G_{0}}(x)$. By Lemma \ref{Lem_CentralizerOfFieldAutOfPSL_2(q)}, we have $M\cong \PGL_{2}(q^{1/2})$. We can  view $x\in \mathrm{Aut}(G_{0})$ as an automorphism of $\PGL_{2}(q) \cong G_{0}^{\ast}\trianglelefteq \mathrm{Aut}( G_{0})$. Since $| \PGL_{2}(q):M|$ is even and $\PGL_{2}(q))$ has trivial centre, there exists $g\in \PGL_{2}(q)  \backslash M$ such that $x$ normalizes but does not centralize $M^{g}$ by Lemma \ref{Lem_CentralizerCosetArgument}. Hence, the pair $(M^{g},x)$  contradicts the minimality of $(G,x)$ for $q^{1/2}\geq  4$. Otherwise, $G_{0}=\PSL_{2}(4)\cong A_{5}$, or $G_{0}=\PSL_{2}(9)\cong A_{6}$, which are contained in case (A).
  
(B.3) \emph{Graph involutions.}\\
  Suppose that $x$ belongs to a $G_{0}^{\ast}$-coset of $\Aut(G_0)$ represented by a graph automorphism. By our conventions (see    \cite[Definition 2.5.13]{GLS}), there are no graph automorphisms of $F_{4}(2^{a})$, $G_{2}(3^{a})$, $B_{2}(2^{a})$, or of the Suzuki--Ree groups.  This leaves us with $G_{0}\cong \PSL_{n}^{\varepsilon}(q)$ ($  n\geq 3$), $G_{0}\cong E_{6}^{\varepsilon}(q)$, and $ G_{0}\cong D_{n}^{\varepsilon}(q)$ ($n\geq 4$).

 (B.3.1) $G_{0}\cong \PSL_{n}^{\varepsilon}(q)$, where $n\geq 3$.

 (B.3.1.1) $n\geq 5$. We split  the discussion according to the different cases of Lemma \ref{Lem_involutiveGraphAutAnepsilon}.
 
 (i) If $n$ is odd, then both $\iota $ ($\varepsilon =1$) and $\varphi _{q}$ ($\varepsilon =-1$) normalize but do not centralize a type $\PSL_{n-1}^{\varepsilon}(q)$ subgroup, and $(G,x)$ cannot be a minimal counterexample.
 
(ii) If $n$ is even ($n\geq 6$) and $q$ is even, then $\iota $ ($\varepsilon  =1$), $\varphi _{q}$ and $\varphi _{q}x_{0}(1)$ ($\varepsilon  =-1$), normalize but do not centralize a type $\PSL_{n-1}^{\varepsilon}(q)$ subgroup,  while for $x=\iota St$ ($\varepsilon =1$), $x$ normalizes and doesn't centralize a type $\PSL_{n-2}(q)$ subgroup.

 (iii) If $n$ is even ($n\geq 6$) and $q$ is odd, then, for $\varepsilon =1$, $x$ normalizes but does not centralize a type $\PSL_{n-2}(q)$ subgroup in all three cases. For $\varepsilon =-1$, $x = \varphi_q$ acts nontrivially on a type $\PSU_{n-1} (q)$ subgroup in all three cases.

 (B.3.1.2) $G_0 \cong \PSL_{4}^{\varepsilon}(q)$. If $q$ is even, then we can repeat the argument of (B.3.1.1)(ii),  unless $q=2$, $\varepsilon=+$ and  $x=\iota St$. But then $\PSL_4(2) \cong A_8$ and we have eliminated this case already. If $q$ is odd, then the discussion of  (B.3.1.1)(iii) applies for $\varepsilon =-1$. For $\varepsilon =1$, $\iota S^{+}$ and  $\iota S^{-}$  normalize the subgroup $\{ \diag[A,1,1] \mid A \in \GL_2(q) \} \le \GL_4(q)$. Moreover, $\iota S^{+}$ and $\iota S^{-}$  restrict to $\iota S_{0}$ where 
\[S_{0}=\left( 
\begin{array}{cc}
0 & 1 \\ 
1 & 0
\end{array}
\right).\]
  A straightforward calculation shows that $\iota S_{0}$ coincides  with an involution in $\PGL_{2}(q)$, hence $(\PSL_{2}(q),\iota S_{0})$ contradicts the minimality of $  (G,x)$ unless $q=3$, in which case $\PSL_{4}(3) \in CVL2$. In summary, if $n=4$, then it remains to check the claim for $G_{0}=\PSL_{4}(q)$, $q$ odd and $x=\iota S$.

 (B.3.1.2.1) $G_0 \cong \PSL_{4}(q)$, $q \equiv  1 \pmod{4}$, $x = \iota S$. First we determine $C_{\PGL_4(q)}(\iota S)$. Let $A\in  \PGL_{4}(q)$. Then $A \in C_{\PGL_4(q)}(\iota S)$  if and only if $A^{\iota S}=cA$  for some scalar $c$. One checks that this is equivalent to $  AS^{-1}A^{T}=c^{-1}S^{-1}$, and since $S^{-1}$ is the matrix of an alternating form, we have $A\in C_{\PGL_4(q)}(\iota S)$ if and only if $A\in \PGSp_{4}(q)$,  hence $C_{\PGL_4(q)}(\iota S) \cong \PGSp_{4}(q)$. Now consider $C_{G_{0}}(\iota S)=G_{0}\cap C_{\PGL_4(q)}(\iota S)$.  If  $C_{G_{0}}(\iota S)= C_{\PGL_4(q)}(\iota S)$, then $\PSL_4(q)$ contains $\PGSp_4(q)$, which is inconsistent with our assumption that $q \equiv 1 \pmod{4}$ (see \cite{kthesis} for example). Moreover, the fact that $q\equiv 1 \pmod{4}$ implies that $G_0$ has index $4$ in $\PGL_4(q)$. Let $K= C_{\PGL_4(q)}(\iota S) \cong \PGSp_4(q)$. Using the second isomorphism theorem, we have $| G_{0}K:G_{0}| =| K:K\cap G_{0}| \in  \{ 2,4\}$. Since $K$ is almost simple with socle $\PSp_{4}(q)$ of  index $2$, $| K:K\cap G_{0}| =2$, which implies that $ C_{G_{0}}(\iota S)=\PSp_{4}(q)$. Since $\PSp_{4}(q)$ has trivial centre and $|\PSL_{4}(q) :\PSp_{4}(q)|$ is even, we obtain a contradiction by Lemma \ref{Lem_CentralizerCosetArgument}.

 (B.3.1.2.2) $G_0 \cong \PSL_{4}(q)$, $q \equiv 3 \pmod{4}$, $x = \iota S$. Note that $q=r^{f}$ is not  a square, and hence $f$ is odd. Let $u$ be a primitive prime divisor of $ r^{3f}-1$. By Remark \ref{Rem_PrimitivePrimeDivisor} and \cite[Lemma 2.1]{GM}, either we can take $u \ge 3e+1$ (where $e=3f$), or $u=2e+1$ and $u^2$ divides $q^{3}-1$ (note that since $f$ is odd, $e+1$ is even and therefore $e+1$ cannot be prime).  
  Let $y \in G_0$ of order $u$ or $u^2$ in each case respectively. We will show that $\langle x,y\rangle =\langle  G_{0},x\rangle $. Suppose to the contrary that $\langle  x,y\rangle \leq M<\langle G_{0},x\rangle $ for some maximal subgroup $M$ of $\langle G_{0},x\rangle$. Let $  M_{0}=M\cap G_{0}$. First suppose that $M_{0}$ is reducible and hence (since graph automorphisms are present) of type $P_2$, $P_{1,3}$ or $\GL_1(q) \oplus \GL_3(q)$. But the order of $y$ does not divide $|P_2|$ or $|P_{1,3}|$ and if $x$ normalizes a subgroup of type $\GL_1(q) \oplus \GL_3(q)$, then we can reduce to the case $G_0 = \PSL_3(q)$, which contradicts minimality.  So we may assume that $M_{0}$ is irreducible and now \cite[Theorem 2.2]{GM} shows that there are no such maximal subgroups containing both $x$ and $y$. It now follows that $\langle x,y \rangle = \langle x,G_0 \rangle$ as claimed.


 (B.3.1.3) $G_0 \cong \PSL_{3}^{\varepsilon}(q)$. By Lemma \ref{Lem_involutiveGraphAutAnepsilon} we may assume $x = \iota$ when $\varepsilon=+$ and $x= \varphi_q$ when $\varepsilon=-$. First suppose that $\varepsilon =+$. We claim that $C_{G_{0}}(\iota ) =\PSO_{3}(q) (\cong \SO_3(q))$. To see this, let $g\in  C_{G_{0}}(\iota)$ and  write $g=AZ( \SL_{3}(q))$ for some $A\in \SL_{3}(q)$. Now since $\iota (g) =g$, we have $\iota (A) =\lambda A$ for some $\lambda  \in \mathbb{F}_{q}$ such that $\lambda ^{3}=1$. But $\lambda ^{3}=1$ implies  that $\lambda =(\lambda ^{-1}) ^{2}$ and so for $B=\lambda  ^{-1}A\in \SL_{3}(q)$ it follows that $g= BZ( \SL_{3}(q))$ and $\iota (B) =B$. That is, $B \in \SO_3(q)$  and  since $Z(\SL_3(q)) \cap \SO_3(q) = \{1\}$, we have $C_{\PSL_3(q)}(\iota) = \SO_3(q)$. Now suppose that  $ \varepsilon = -$. Let $g\in C_{G_{0}}(\varphi _{q})$. Then $g=AZ(\SU_{3}(q) )$ where $A\in \SU_{3}(q)$ (in particular we may assume that $\varphi_q(A) = \iota (A)$).  Since $\varphi_{q}(g) =g$, we have $\varphi_{q}(A) =\lambda A$ for some $\lambda \in \mathbb{F}_{q^{2}}$ such that $\lambda ^{3}=1$. As before, $\lambda ^{3}=1$ implies that $\lambda $ has a  square root so we can choose $A$ such that $\varphi _{q}(A) =A$.  But $\varphi_q(A) = \iota(A)$, hence $A= \iota (A)$ and  $A \in \SO_3(q)$. Finally, $\SU_{3}(q)$ contains  $\SO_{3}(q)$ as a subfield subgroup, so since $Z(\SU_3(q)) \cap \SO_3(q) = \{1\}$, we have $C_{G_{0}}(\varphi  _{q}) =\SO_{3}(q)$. Since $| G_{0}:\SO_{3}(q) |$ is even for all $q$ and $\varepsilon $, and since $  Z( \SO_{3}(q) ) =1$, it follows from Lemma \ref{Lem_CentralizerCosetArgument} that $x$ normalizes but does not centralize a  subgroup $M\cong \SO_{3}(q)$, which is almost simple for $q\geq  4$. So we may assume that $q\le 3$, in which case $G_{0}$ is isomorphic to $\PSL_{3}(2)$, $\PSL_{3}(3)$, or $\PSU_{3}(3)$ ($\PSU_{3}(2)$ is solvable). But all of these groups are contained in $CVL2$.

 (B.3.2) $G_{0}\cong E_{6}^{\varepsilon}(q)$. By Lemma \ref{Lem_involutiveGraphAutE_6}, there are two $G_{0}^{\ast}$-classes of graph  involutions in all subcases. These classes have representatives of the form $z$ and $  zt$, where $z$ is a standard graph automorphism ($z=\gamma $ for $\varepsilon =1$ and $z=\varphi _{q}$ for $\varepsilon =-1$) and $t$ is an  inner automorphism such that $zt=tz$ (see \cite[(19.7)]{AS} for $q$ even and \cite[Table 4.5.1 and p. 157]{GLS} for $q$ odd).

 We have $C_{G_{0}}(z)\cong F_{4}(q)$  for both values of $\varepsilon$ (see \cite[Table 4.5.1]{GLS}  for $q$ odd  and \cite[(19.9)(iii)]{AS} for $q$ even). If $x=zt$, then, since $t$ and $z$  commute, we have $t\in C_{G_{0}}(z)$, and $x$ normalizes $C_{G_{0}}(z)$. If $x$ centralizes $C_{G_{0}}(z)$, then $t\in  Z(C_{G_{0}}(z)) =Z(F_4(q))$, which is a  contradiction since $F_4(q)$ has trivial centre. If $x=z$, then it is easily verified that $|E_{6}(q)| /| F_{4}(q)|$ and $|{^2}E_{6}(q) |/| F_{4}(q)|$ are  even and so, by Lemma \ref{Lem_CentralizerCosetArgument}, there exists a  conjugate of $F_4(q)$ that is normalized but not centralized by $x$.  Hence $(G,x)$ cannot be a minimal counterexample. 
 
 (B.3.3) $G_{0}\cong \mathrm{P}\Omega _{d}^{\pm}(q)$, $d\geq 8$. Let $(V,Q)$ be the $d$-dimensional orthogonal space on which $\GO_{d}^{\pm}(q)$ acts naturally ($Q$ is a  nondegenerate quadratic form on $V$). Then $x\in \PCO_{d}^{\pm}(q)$ ($\PGO_{d}^{\pm}(q)$ in the notation of \cite[p. 70--72]{GLS}). By \cite[Proposition  1.4(a)]{monodromy}, $x$ stabilizes a nontrivial orthogonal decomposition $  V=W\perp W^{\prime}$, since $d>4$. Moreover, suppose that $b=\dim W^{\prime}\geq 1$ is minimal. Again by  \cite[Proposition 1.4(a)]{monodromy}, if $b>4$, then $x$ restricted to $W^{\prime}$ must stabilize a nontrivial orthogonal  decomposition $W^{\prime}=W^{\prime \prime}\perp W^{\prime \prime \prime}$  and hence $V=(W+W^{\prime \prime}) \perp W^{\prime \prime  \prime}$ is an orthogonal decomposition stabilized by $x$ and contradicts  the minimality of $b$. Hence $b\leq 4$. Now we prove that among the $x$-invariant  nontrivial orthogonal decompositions $V=W\perp W^{\prime}$ with minimal $  b=\dim W^{\prime}$, there exists at least one for which $x$ acts nontrivially on $W$.  Suppose not. Then $x$ trivially  stabilizes an orthogonal decomposition $W=W_{d-b-1}\perp W_{1}$ where $\dim  W_{i}=i$. Therefore $x$ stabilizes the orthogonal decomposition $  V=(W_{d-b-1}\perp W^{\prime})\perp W_{1}$ and $b=1$. Moreover, since $x$  acts nontrivially on $V$, it acts nontrivially on $W_{d-b-1}\perp  W^{\prime}$, which is a contradiction. Now  $\dim W = d-b$ and since $1\leq b\leq 4$ we have $4\leq d-4\leq d-b\leq d-1$. Thus we can reduce to  the case $\mathrm{P}\Omega _{d-b}^{\pm}(q)$, which is simple unless $d-b=4$ (and $d=8$) and $q=2$ or $3$. But in this case,  $G_0 \cong \mathrm{P}\Omega _{8}^{\pm}(2)$, or $\mathrm{P}\Omega _{8}^{\pm}(3)$, which  are contained in $CVL2$.

(B.4) \emph{Graph-field involutions.} \\
 Suppose that $x$ belongs to a $G_{0}^{\ast}$-coset of $\Aut   (G_0)$ represented by a graph-field  automorphism. In particular, $G_{0}$ is untwisted. By Lemma \ref{Lem_FieldAutCoset} we  may assume that $x$ is a standard graph-field automorphism of $G_{0}$ (that  is, $x\in \Phi _{G_{0}}\Gamma _{G_{0}}-\Phi _{G_{0}}-\Gamma _{G_{0}}$ unless $G_{0} \cong B_{2}(2^a)$, $  F_{4}(2^a)$, or $G_{2}(3^a)$ in which case $x\in \Phi _{G_{0}}\Gamma _{G_{0}}-\Phi _{G_{0}}$).

 (B.4.1) $G_{0}$ is untwisted and  $G_0 \not\cong B_{2}(2^a)$, $  F_{4}(2^a)$, or $G_{2}(3^a)$. By \cite[Theorem  2.5.12(e)]{GLS}, $x=x_{\Phi}x_{\Gamma}$ where $x_{\Phi}\in \Phi  _{G_{0}}$ and $x_{\Gamma}\in \Gamma _{G_{0}}$ are commuting involutions. In  particular, $q$ is a square and we let $q_{0}=q^{1/2}$. Since $\Gamma _{G_{0}}\neq 1$, $G_{0}$ is of type $A_{n}$ ($n\geq 2$), $D_{n}$ ($n\geq 4$), or $E_{6}$. In each case, $x$ acts noncentrally on a  subgroup of type $A_n(q_0)$, $D_n(q_0)$ or $E_6(q_0)$ respectively and thus $(G,x)$ cannot be a minimal counterexample. 


 (B.4.2) $G_{0} \cong B_{2}(2^{a})$ ($a\geq 2$), $F_{4}(2^{a})$ or $G_{2}(3^{a})$. By  \cite[Theorem  2.5.12(e)]{GLS}, $\Phi _{G_{0}}\Gamma _{G_{0}}$ is cyclic and $  | \Phi _{G_{0}}\Gamma _{G_{0}}:\Phi _{G_{0}}| =2$. Thus $ x^{2}=\varphi _{r}$. But $x^{2}=1$ so $a=1$, and $F_{4}(2)$, $G_{2}(3) \in CVL2$.

(C) \emph{Sporadic groups.}

 Note that $| \mathrm{Out}(G_0) | \in  \{ 1,2\}$ so either $G=G_{0}$ or $G=G_0:2$. We obtain a contradiction in one of the following ways.

(C.1) By applying Lemma \ref{Lem_SingleAut(G)classOf involutions} when there is a single $G$-class of involutions. 

 (C.2) Using \cite{Atlas}, we look for a prime divisor $p$ of $  | G|$ such that for every maximal subgroup $M$,  if $p$ divides $| M|$, then $M$ is almost simple or of odd  order. If such a  prime exists, then we let $y\in G$ be a $p$-element. Then either $\langle x,y\rangle =G$ is nonsolvable or $\langle x,y\rangle \le M$ for an almost simple subgroup $M$ in which case $(G,x)$ is not a minimal counterexample.

(C.3) For the remaining cases we use \textsc{Magma}.
\end{proof}

\subsection{Proof of Theorem \protect\ref{Th_3plus2}}

 In this section, we prove that if $G$ is  almost simple and $x \in G$  has order $3$, then there exists a $2$-element $y \in G$ such that $\langle x,y \rangle$ is not solvable. 
\begin{definition}
 The list $CVL3$ consists of the following groups: $\PSU_{3}(3)$, $\PSL_{3}(3)$, $\PSp_{4}(3)$, $G_{2}(3)$, $\PSU_{4}(3)$, $^{3}D_{4}(3)$, ${^{3}}D_{4}(2)$, $\PSL_{4}(2)$, $\PSU_{6}(2)$, $\PSU_{4}(2)$, $\PSL_{3}(4)$, $\PSU_{3}(4)$, $\PSL_{2}(8)$, $\PSL_{2}(27)$, ${^{2}}B_{2}(8)$, $  D_{4}(2)$. For each $G_0\in CVL3$ we used \textsc{Magma} to verify that if $x\in \Aut (G_0)$ has order $3$, then there exists a $2$-element $y\in G_0$ such that $\langle  x,y\rangle $ is not solvable. 
\end{definition}

\begin{proof}[Proof of Theorem \protect\ref{Th_3plus2}]
 Let $(G,x)$ be a minimal counterexample to the claim of Theorem \ref{Th_3plus2}. We split the discussion according to the possibilities for $ G_{0}$.

(A) \emph{Alternating groups.} \\
 Suppose that $G_{0}\cong A_{n}$ and $n\geq 5$. Since $| G:G_{0}|$  is not divisible by $3$ we may assume that $G=G_{0}$. We may assume $x=(1,2,3)$ or $x=(1,2,3) (4,5,6) x_{1}$ where $  x_{1}\in \mathrm{Sym}\{ 7,\ldots,n\}$ and $x_{1}^{3}=1$. For $x=(1,2,3)$ let $y=(1,4) (3,5)$, and for $  x=(1,2,3) (4,5,6) x_{1}$ let $y=(1,2)(3,4)$. Then it is straightforward to check that $\langle x,y\rangle $ is nonsolvable and we have a contradiction.

(B) \emph{Simple groups of Lie type of characteristic $r$.}

(B.1)  $x \in G_{0}^{\ast} = \Inndiag(G_0)$.

 (B.1.1) $G_{0}\cong \PSL_{2}(q)$, $q\geq 4$. Here $G_{0}^{\ast}\cong  \PGL_{2}(q)$, and since $| \PGL_{2}(q):\PSL_{2}(q)|=(2,q-1)$, we have $x\in \PSL_{2}(q)$. There is  only one $\PGL_{2}(q)$-class of order $3$ elements in $\PSL_{2}(q)$, and since $\PSL_{2}(q)$ is generated by an order $3$ element and an involution unless $  q=9$ (\cite[Theorem 6]{Mac}), we may assume that $q=9$. However, $\PSL_{2}(9)\cong A_{6}$ is excluded by case (A). 
 
 (B.1.2) \emph{Unipotent elements ($r=3$).}  Since $|   \mathrm{Outdiag}(G_{0})|$ is not divisible by $r$, we may assume that $  G=G_{0}$. 
 Applying the same argument used in Section \ref{Sect_ProofFirstTho(x)=2}  
 case  (B.1.1), most possibilities for $G_{0}$ are ruled out by choosing appropriate  pairs of parabolic subgroups. We are left with $\PSL_{2}(3^{a})$ ($a\geq 2$), ${^{2}}G_{2}(3^{a})$ ($a\geq 3$ odd), $\PSU_{3}(3^{a})$ ($a\geq 1$), $\PSL_{3}(3)$, $\PSp_{4}(3)$, $G_{2}(3)$, $\PSU_{4}(3)$, or $^{3}D_{4}(3)$. The case $G_{0}\cong \PSL_{2}(3^{a})$ is eliminated in (B.1.1). We now consider the remaining groups.

 (B.1.2.1) $G_{0}\cong {^{2}}G_{2}(3^{a})$, ($a\geq 3$ odd). By \cite[p. 87]{ward}, $G_{0}$ has three conjugacy classes of  elements of order $3$ with representatives $X$, $T$ and $T^{-1}$. We have  $|C_{G_{0}}(T) | =2q^{2}$ (see \cite[III-2 p.78]{ward}), so both $T$ and $T^{-1}$ centralize an involution. There is a single  class of involutions \cite[p. 63]{ward} with centralizers isomorphic to  $2\times \PSL_{2}(q)$ \cite[property III, p.62]{ward}. Thus  $T$ and $T^{-1}$ belong to a $\PSL_{2}(q)$ subgroup contradicting the minimality of our counterexample. The centralizer of $X$ is a Sylow $3$-subgroup of order $3^{3a}$ \cite[p. 78 (3)]{ward}. By \cite[Theorem 3.3.1(c)]{GLS}, the centre of the  Sylow $3$-subgroup $U = \langle X_{\alpha}  \mid \alpha \in \Pi \rangle$ of $G_{2}(3^{a})$ (untwisted) is $Z(U)  =X_{-\alpha _{\ast}}X_{\alpha _{s}}$, where $-\alpha _{\ast}=2\alpha  _{1}+3\alpha _{2}$, and $\alpha _{s}=2\alpha _{1}+\alpha _{2}$ is the  highest short root of $G_{2}$. The graph automorphism of $G_{2}(3^{a})$ normalizes $U$ and interchanges the two root subgroups $  X_{-\alpha _{\ast}}$ and $X_{\alpha _{s}}$. Moreover, it fixes $X_{1}=   x_{2\alpha _{1}+3\alpha _{2}}(1)x_{2\alpha _{1}+\alpha _{2}}(1)$. Thus $  X_{1}\in {^{2}}G_{2}(3^{a})$, and in fact $X_{1}$ belongs to the centre of  the ${^{2}}G_{2}(3^{a})$ Sylow $3$-subgroup contained in $U$. By \cite[p.78, (3)]{ward}, $X_{1}$ has order $3$. Thus we may assume that $X=X_{1}$ and now it is clear that $X$ belongs to a $^{2}G_{2}(3)\cong \PSL_{2}(8):3$ subgroup of $G_{0}$, contradicting the minimality of our counterexample. 
 
 (B.1.2.2) $G_{0}\cong \PSU_{3}(3^{a})$ ($a\geq 1$). Let $V$ be the natural module for  $\GU_{3}(3^{a})$. Since $x\in G$ is unipotent, its Jordan  normal form is either $J_{2} +  J_{1}$ or $J_{3}$. In the first case $x$ is a transvection, and acts as a non-scalar on a two-dimensional  nondegenerate subspace $W$ of $V$. Therefore we can reduce to the case $G_0 \cong \PSL_2(3^a)$ when $a \ge 2$. Now suppose that $x$ has Jordan normal form $J_{3}$. By  \cite[Proposition 4.5.5(II)]{KL}, $\PSU_{3}(3^{a})$ has an $\SO_{3}(3^{a})$ subgroup and $\SO_{3}(3^{a})$ contains an element $x_{1}$ with a single $J_{3}$ block (by \cite[Section 3.3]{Bur2} for example). Now $x_{1}$ is  $\GU_{3}(q)$-conjugate  to $x$ since  two elements in $\GU_{3}(q)$ are  conjugate if and only they have the same Jordan form. Thus $(G,x)$ cannot be a minimal counterexample for $a\geq 2$. If $a=1$, then we have $\PSU_{3}(3)\in CVL3$. 
 
 (B.1.2.3) The remaining groups $G_0$ satisfy $G_{0}\cong \PSL_{3}(3)$, $\PSp_{4}(3)$, $G_{2}(3)$, $\PSU_{4}(3)$, ${^3}D_{4}(3)$ and  are contained in  $CVL3$.

  (B.1.3) \emph{Semisimple elements ($r\neq 3$) $G_{0}$ exceptional  or a  group of type $D_{4}$}. 
  By Lemma \ref{Lem_semisimpleInParabolic} and the fact that $O^{r^{\prime}}( C_{G^{\ast}}(x) )$ is nontrivial (see \cite[Table 4.7.3.A]{GLS}), $x$ belongs to a  maximal parabolic subgroup of $G$. Therefore we can apply Lemma \ref{lem:2.2}(b), and arguing as in Case (B.1.2) of Section \ref{Sect_ProofFirstTho(x)=2},  
 we eliminate all possible $G_{0}$ except ${^{2}}F_{4}(2)^{\prime}$, 
  $G_{2}(2)' \cong \PSU_{3}(3)$ and ${^{3}}D_{4}(2)$. Now ${^{2}}F_{4}(2)^{\prime}$ has only one class of elements of order $3$ and has standard generators $a,b$ of order $2$ and $3$ respectively. Hence $x$ is conjugate to $b$ and ${^2}F_{4}(2)^{\prime}$  cannot be a minimal counterexample. The remaining groups $G_{2}(2)' \cong \PSU_{3}(3)$ and ${^{3}}D_{4}(2)$  are contained in $CVL3$.

(B.1.4) \emph{Semisimple elements ($r\neq 3$), $G_0$ classical}. 
  By  \cite[Lemma 3.11]{Bur2}, 
 $x$ lifts to an element of $\widehat{G}$, which  we continue to call $x$, where $\widehat{G}\cong  \GL_{n}^{\varepsilon}(q)$ or $\widehat{G}\cong \Sp_{n}(q)$ or $\widehat{G}\cong \Omega _{n}^{\varepsilon}(q)$, and $o(x) =3$ unless $G_{0}\cong \PSL_{n}^{\varepsilon}(q)$, and $3|(q-\varepsilon,n)$, in which case we may assume that $o(x) =3^{k}$. Let $V$ be the natural $\widehat{G}$-module of dimension $n$. Since $x$ is semisimple, $V$ decomposes as a sum of $x$-modules
\begin{equation} \label{eqn:decomp}
   V= U_1 \perp \cdots \perp U_k \perp (W_{k+1}\oplus W_{k+1}') \perp \cdots \perp (W_{k+s} \oplus W_{k+s}'),
\end{equation}
where $x$ acts irreducibly on each $U_i$, $W_i$ and $W_i'$, the $W_i$ and $W_i'$ are totally singular, and  the $U_i$ and $(W_i \oplus W_i' )$ are nondegnerate (except in the linear case of course). 
  If $x$ lifts to an element of order $3$ in $\widehat{G}$, then the dimension of each of these modules is at most $2$ since $3|q^2-1$ (and therefore $\GL_m(q^k)$ does not contain elements of order $3$ acting irreducibly when $ m\ge 3$). In the rest of this section,  $X$ will stand for one of the irreducible subspaces $U_i$, $W_i$ or $W_i'$ in \eqref{eqn:decomp}. We note that if $x$ acts trivially on $X$, then $\dim X = 1$ since we are assuming that $X$ is an irreducible $x$-module.


 (B.1.4.1) $G_{0}\cong \PSL_{n}(q)$.  If $x$ lifts to an element of order $3$ in $\widehat{G}$ and $q\geq 4$, then minimality of $(G,x)$  implies $n=2$ and we are done by case (B.1.1). If $q \le 3$, then $q=2$ since $r\neq 3$. In this case, since $\PSL_{2}(2)$ is  solvable, we have the following possibilities: If there exists $X$ in \eqref{eqn:decomp} of  dimension $1$, then, by minimality, $G_{0}\cong \PSL_{3}(2)\cong  \PSL_{2}(7)$ (which is eliminated in (B.1.1)), and if   every $X$ in \eqref{eqn:decomp} is $2$-dimensional, then minimality implies that $G_{0}\cong \PSL_{4}(2) \cong A_8$.  If $x$ does not lift to an element of order $3$ in $\widehat{G}$, then   all of the $X$ in \eqref{eqn:decomp} are $3$-dimensional (see the proof  of \cite[Lemma 3.11]{Bur2} for example)  and it follows  that $n=3$ by minimality of $(G,x)$. Thus it remains to check $G_0=\PSL_{3}(q)$ where $(3,q-1)=3$, $x$ has order $3^k$, $x$ acts irreducibly and  and $x^3$ acts as a scalar (see (B.1.4.5) below).

(B.1.4.2) $G_{0}\cong \PSU_{n}(q)$.

(B.1.4.2.1) If $x$ does not lift to an element of order $3$, then, as in (B.1.4.1), all of the modules in \eqref{eqn:decomp} are $3$-dimensional and it follows by minimality that either $n=3$, $q \ge 4$, $(3,q+1)=3$ and $x$ acts irreducibly (see (B.1.4.5) below) or $q=2$ and $G_0 = \PSU_6(2) \in CVL3$ ($\PSU_3(2)$ is solvable). 


 (B.1.4.2.2)  If $x$ lifts to an element of order $3$, then we have $\dim X \le 2$ for all $X$ in \eqref{eqn:decomp}. If $  q\geq 4$, then minimality implies that $n=2$ and we are done by case (B.1.1) since $\PSU_{2}(q)\cong \PSL_{2}(q)$. If $q=2$, then it suffice to check $G_{0}\cong \PSU_{4}(2)$, which is contained in $CVL3$.

(B.1.4.3) $G_{0} \cong \PSp_n(q)$ ($n \ge 4$). Since $x$ lifts to an element of order $3$  we have $\dim X \le 2$ for all $X$  in \eqref{eqn:decomp}. 
First suppose that there is a $2$-dimensional $X$ in \eqref{eqn:decomp}  (on which $x$ necessarily acts nontrivially). 
If $X$ is totally singular and $q \ge 4$, then we can reduce to the case of $\PSL_2(q)$ by restricting to $X$.  If $X$ is totally singular and $q=2$, then we can reduce to the case of $\PSp_4(2)$ by restriction to $X \oplus X'$.
Thus we may now assume that all of the $2$-dimensional spaces in \eqref{eqn:decomp} are nondegenerate. If $q \ge 4$, then we can reduce to the case $\PSL_2(q)$ by restricting to $X$. If $q=2$, then we can reduce to the case of $\PSp_4(2)$ by restricting to $X \perp Y$, where $Y$ is a nondegenerate subspace in \eqref{eqn:decomp} (necessarily of dimension $2$) or is a sum $W_i \oplus W_i'$ of totally singular $1$-spaces in \eqref{eqn:decomp}.

 Finally,  it remains to consider the case where all of the $X$ are $1$-dimensional (and hence totally singular).  Choose $X$ on which $x$ acts nontrivially. Again we can reduce to the case of $\PSL_2(q)$ by restriction to $X \oplus X'$ when $q \ge 4$ (note that since $X \oplus X'$ is nondegenerate,  $x$ does not act as a scalar on $X \oplus X'$ since $Z(\Sp_2(q)) = \langle \pm I_2 \rangle$). 
  If $q=2$, then we reduce to the case of  $\PSp_4(2) \cong S_6$ by restricting to $(X \oplus X' ) \perp (Y \oplus Y')$ where $Y$ is another totally singular $1$-space in \eqref{eqn:decomp}.  Thus in all cases, we have shown that if $G_0 = \PSp_n(q)$ is  a minimal counterexample, then $(n,q)= (4,2)$ and $\PSp_n(q) \cong S_6$; but this case is  eliminated in (A).



(B.1.4.4) $G_0 \cong \mathrm{P}\Omega^{\varepsilon}_n(q)$ ($n \ge 7$)
Since $x$ lifts to an element of order $3$  we have $\dim X \le 2$ for all $X$  in \eqref{eqn:decomp}. 
If there is a totally singular $2$-dimensional $X$ in \eqref{eqn:decomp}, then we can reduce to the case $\PSL_2(q)$ for $q \ge 4$ by restricting to $X$. If $q=2$, then we note that $\Omega^{\varepsilon}_k(2)$ is solvable for $k \le 4$;  so choose another irreducible subspace $Y$ in \eqref{eqn:decomp}. If $Y$ is nondegenerate, then we reduce to $\mathrm{P}\Omega_5(2)$ or $\mathrm{P}\Omega^{\pm}_6(2)$ by restricting to $(X \oplus X') \perp Y$ while if $Y$ is totally singular, then  we can reduce to  $\PSL_3(2)$ or $\PSL_4(2)$ by restricting to $X \oplus Y$. If there is a $2$-dimensional nondegenerate $X$ in \eqref{eqn:decomp}, then let $\widetilde{X}$ denote the sum of the remaining subspace in \eqref{eqn:decomp}.  We can reduce to the case $\mathrm{P}\Omega^{\varepsilon}_{n-2}(q)$ by restricting to $\widetilde{X}$  if $x$ acts nontrivially on $\widetilde{X}$. If $x$ acts trivially on $\widetilde{X}$, then we can reduce to the case of $\mathrm{P}\Omega_5(q) = \PSp_4(q)$ (by restricting to $X \perp Y$ where $Y$ is any $3$-dimensional nondegenerate subspace of $\widetilde{X}$). It remains to consider the case where each $X$ is $1$-dimensional. Similarly, we can reduce to the case of $\mathrm{P}\Omega^{\varepsilon}_{n-1}(q)$ or $\mathrm{P}\Omega^{\varepsilon}_{n-2}(q)$ and since $n\ge 7$ it is clear that $(G,x)$ cannot be a minimal counterexample here either.

 (B.1.4.5) $G_{0}\cong \PSL_{3}^{\varepsilon}(q)$, with $(3,q-\varepsilon)=3$, $x$ irreducible and $x$ does not lift to an order $3$ element in $\GL^{\varepsilon}_3(q)$. If this is the case, then $x^3= \lambda I_3$ for some $\lambda \in \mathbb{F}_{q^u}$ ($u=1$ for $\varepsilon=1$, otherwise $u=2$). We note that $\lambda$ has no cube root in $\mathbb{F}_{q^u}$ otherwise we can choose $\mu \in \mathbb{F}_{q^u}$ such that $\mu^3= \lambda^{-1}$, in which case  $(\mu x)^3 =1$ and $\mu x$ is a lift of $x$ of order $3$. In particular, $x$ has rational canonical form
   \[\left(\begin{array}{ccc}0 & 0 & \lambda \\1 & 0 & 0 \\0 & 1 & 0\end{array}\right)\]
and $\langle x, \PSL_3^{\varepsilon}(q) \rangle = \PGL_3^{\varepsilon}(q)$ since we cannot write $x = x_0 (\mu I_3)$ for $x_0 \in \SL^{\varepsilon}_3(q)$ and $\mu I_3 \in Z(\GL_3^{\varepsilon}(q))$  (for $\det x = \lambda$, $\det (\mu x_0) = \mu^{3}$ and $\lambda$ has no cube root in $\mathbb{F}_{q^u}$).


 Now we choose $y\in G_{0}$. If $r=2$, then $y$ is chosen so that its preimage  in $\GL_{3}^{\varepsilon}(q)$ has a Jordan normal form $J_{3}$  (a regular unipotent element of order $4$).   If $r>3$ we choose $y$ as follows. For $\varepsilon =1$, take $y$ to be the $ 2$-part of the image in $\PSL_{3}(q)$ of $\diag[A, 1/\det (A)] \in \SL_3(q)$
 where $A\in \GL_{2}(q)$ is  a Singer cycle (of order $q^{2}-1$). For $\varepsilon =-1$, take $y$ to be  the $2$-part of the image in $\PSU_{3}(q)$ of $\diag[  a,a^{-q},a^{q-1}] \in \SU_{3}(q)$, where $a\in  F_{q^{2}}^{\times}$ has order $q^{2}-1$. In both cases $o(y) =(q^{2}-1)_{2} \geq 8$.

 Let $M_{1},\ldots,M_{k}$ be the maximal subgroups of $G = \langle x, \PSL_3^{\varepsilon}(q) \rangle = \PGL^{\varepsilon}_3(q)$ containing $y$. We aim to   show  that $|x^{G}| >|x^{G}\cap \bigcup_{i=1}^{k}M_{i}|$, in which case there exists $g\in  G $ such that $x^{g}\notin \bigcup_{i=1}^{k}M_{i}$. It would then follow that $\langle  x^{g},y\rangle= \langle x,y^g \rangle =G$ and hence that $(G,x)$ is not a minimal counterexample.

  Let $Y_{1},\ldots,Y_{l}$ be representatives  of distinct conjugacy classes of $M_{1},\ldots,M_{k}$ and  let $n_{i}$ be the number of distinct conjugates of $Y_{i}$ containing $y$.  First note that 
\begin{equation*}
|x^{G}\cap \bigcup_{i=1}^{k}M_{i}|\leq \sum_{i=1}^{k}|x^{G}\cap
M_{i}|=\sum_{i=1}^{l}n_{i}|x^{G}\cap Y_{i}|.
\end{equation*}
By double counting, we have $n_{i}=| y^{G}\cap Y_{i}|| G:N_{G}(Y_{i}) |/ |y^{G}|$.
Therefore the inequality $|x^{G}| >|x^{G}\cap \bigcup_{i=1}^{k}M_{i}|$ will  follow if 
\begin{equation}
|x^{G}|>\sum_{i=1}^{l}\frac{| y^{G}\cap Y_{i}|}{
| N_{G}(Y_{i}) |}| C_{G}(y) | |x^{G}\cap Y_{i}|\text{.}
\label{Eq_CountingArgumentInequality}
\end{equation}


We aim to bound above the right-hand side of \eqref{Eq_CountingArgumentInequality}
by a quantity less than $|x^G|$.
 Since $x$ is semisimple it can be diagonalized over the  algebraic closure of $\mathbb{F}_{q}$. Moreover, $x$ acts  irreducibly on $V$. These two facts imply that  $|C_{G}(x)|=(q^{2}+\varepsilon q+1)$   and $|x^{G}| = q^{3}(q- \varepsilon)(q^2-1)$.  




Using the lists of maximal subgroups of $\PGL_{3}^{\varepsilon}(q)$ in \cite{BHR} (see also \cite[Theorem 6.5.3]{GLS}) and since we may assume that each $Y_i$ contains elements of order $o(y)$ and must not be almost simple (by minimality of $G$), we find in each case that there is at most one possible isomorphism type for $Y_i$. Also, it is easy to see that $N_G(Y_i)=Y_i$ since these subgroups are maximal subgroups of $G$.



 If $r=2$ and $q=4$, then we have $\PSL_{3}(4)$, $\PSU_{3}(4)\in CVL3$, so suppose that $r=2$ and $q \ge 8$.  The only irreducible maximal subgroups $Y_i$ containing $y$ (a regular unipotent element) that are not almost simple are isomorphic to $\PGU_3(2)$ in the unitary case and there are no such groups in the linear case (see also \cite[Proposition 4.5.3]{KL}). There are at most three such $G$-classes of maximal subgroups by  \cite[Proposition 4.5.3]{KL}  and $|x^G \cap Y_i|$ is at most $80$ (the number of order $3$ elements in $\PGU_3(2)$). It is well known that $|C_{\PGU_3(q)}(y)| = q^{2}$ and so  the right-hand side of \eqref{Eq_CountingArgumentInequality} is at most $240q^{2}$ and therefore \eqref{Eq_CountingArgumentInequality} holds since $q \ge 8$.

Now suppose that $r \ge 5$ and recall that $o(y) = (q^{2}-1)_2 \ge 8$. Here the only possibility is that $Y_i$ is of type $(q-\varepsilon )\wr S_{3}$. 
 We make the crude estimate $|x^G \cap Y_i| \le |Y_i|~=6(q-\varepsilon )^{2}$ and note that $  |C_{\PGL_{3}^{\varepsilon}(q)}(y)|= q^{2}-1$  (since the eigenvalues of $y$  are distinct and so its centralizer in $\PGL^{\varepsilon}_{3}(q)$ is a maximal torus). By \cite[Proposition 4.2.9]{KL}, there is only one $G$-class of such maximal subgroups. Thus the right-hand side of  \eqref{Eq_CountingArgumentInequality} is at most $6(q-\varepsilon )^{2}(q^{2}-1)$, which is less than $|x^G|=q^{3}(q^{2}-1) (q-\varepsilon )$ since $q\geq 5$ (for $\varepsilon=+$ and $-$). That is,  \eqref{Eq_CountingArgumentInequality} holds and $(G,x)$ is not a minimal counterexample in this case either.

 (B.2) \emph{Field automorphisms}. Suppose that $x$ is a field automorphism. Since $x$ has order $3$, we have $q=r^{a}$ where $a$ is divisible by $3$ and in particular $q\geq 8$. Let $q_{0}=q^{1/3}$. As in Case (B.2) of Section \ref{Sect_ProofFirstTho(x)=2},
  choosing a suitable $\SL_{2}(q)$  subgroup of $G_{0}$, eliminates all cases except when  $G_0$ is $\PSL_{2}(q)$ or a   Suzuki--Ree group.

 (B.2.1) $G_{0}\cong \PSL_{2}(q)$. By Lemma \ref{Lem_CentralizerOfFieldAutOfPSL_2(q)}(i) we have $C_{G_{0}}(x)  \cong \PSL_{2}(q_{0})$. Therefore $3$ divides $|  G_{0}:C_{G_{0}}(x) |$ and $  Z( \PSL_{2}( q_{0}) ) =1$ for all $q_{0}$. Hence, by  Lemma \ref{Lem_CentralizerCosetArgument} there exists a $G$-conjugate of $  C_{G_{0}}(x)$  that is normalized but not centralized by $x$, and  $(G,x)$ cannot be a minimal counterexample for $q_{0}\geq 4$. If $q_0 \le 3$, then we have $G_0 \in CVL3$.

 (B.2.2) $G_{0}\cong {^{2}}F_{4}(2^{a})$ ($a\geq 3$ odd). By \cite[Main Theorem]{2F4max}, $G_{0}$ has a maximal subgroup $\PGU_{3}(2^{a}):2$. By the corollary to the main theorem of \cite{2F4max}, $x$  normalizes and does not centralize such a maximal subgroup and thus $(G,x)$ cannot be a minimal counterexample. 
 
 (B.2.3) $G_{0}\cong {^{2}}G_{2}(3^{a})$ ($a\geq 3$ odd). By \cite[Proposition 4.9.1]{GLS}, we have $O^{3^{\prime}}( C_{G_{0}}(x)) \cong { ^{2}}G_{2}(q_{0})$ and $
 C_{G_{0}}(x) \cong ( O^{3^{\prime}}(C_{G_{0}}(x))) ^{\ast}\cong ( {^{2}}G_{2}(q_{0})) ^{\ast}\cong {^{2}}G_{2}(q_{0})$. In particular, let $t\in C_{G_{0}}(x)$ be an involution, so that $t\in G_{0}$ and $x$ centralizes $t$.
Since $x$ acts on $G_{0}$ and $x$ centralizes $t$, it follows that $x$ acts
on $C_{G_{0}}(t)$ and hence also on $O^{3^{\prime}}(C_{G_{0}}(t))$. By \cite[Table 4.5.1]{GLS}, 
  $O^{3^{\prime}}( C_{G_{0}}( t)) \cong \PSL_{2}(3^{a})$. Note that $x$ does not centralize $  O^{3^{\prime}}( C_{G_{0}}( t))$ and so $(G,x)$ cannot be a minimal counterexample. To see this, note that $  C_{G_{0}}(x)\cong {^{2}}G_{2}(q_{0})$ does not contain an element of order $ (3^{a}+1)/2$ (see the list of maximal tori of ${^{2}}G_{2}(q_{0})$ \cite[Section 2.2]{KSprime}), while $\PSL_{2}(3^{a})$ does. 

 (B.2.4) $G_{0}\cong {^{2}}B_{2}(2^{a})$ ($a\geq 3$ odd). We derive a  contradiction by proving the existence of an involution $y\in G_{0}$  satisfying $\langle x,y\rangle = \langle  G_{0},x\rangle$ (under the assumption that $(G,x)$ is a minimal counterexample).  We bound the number $  | \Gamma |$ of involutions $y\in G_{0}$ such that $\langle x,y\rangle \neq \langle G_{0},x\rangle $ as in \cite[Lemma 3.12]{Guest2}:
\begin{align*}
| \Gamma | \leq & (q_{0}^{3}-1) (q_{0}^{2}+1) + q_{0}^{2}(q_{0}^{3}-1)(q_{0}^{2}+1) /2 +\\
&2q_{0}^{2}(q_{0}^{3}+( 2q_{0}^{3}) ^{1/2}+1)(q_{0}+( 2q_{0}) ^{1/2}+1) (
q_{0}-1) +q_{0}^{2}(q_{0}^{2}+1) (q_{0}^{2}-1).
\end{align*}
 The total number of involutions in $G_{0}$ is $(q^{2}+1) (q-1)$  (see \cite[Propositions 1 and 7]{2B2}). The  right-hand side of the last inequality is less than $(q^{2}+1) (q-1)$ unless $q_{0}=2$. If $q_{0}=2$, then  we have $G_0 \cong {^{2}}B_{2}(8)\in CVL3$.

 (B.3) \emph{Graph automorphism}. Suppose that $x$ belongs to a $G_{0}^{\ast}$-coset of $\Aut (G_0)$ represented by a graph automorphism. We can argue in a similar way as in Case (B.3) of Section \ref{Sect_ProofFirstTho(x)=2} to
 eliminate all cases except $G_{0}=D_{4}(q)$ and $G_{0}={}^{3}D_{4} (q)$ where $q=q_{0}^{3}$. 

 (B.3.1) $G_0 \cong D_{4}(r^{a})$ or ${^3}D_{4}(r^a)$, $r\neq 3$. By \cite[Proposition 1.4.1]{D4max} and \cite[Main Theorem]{3D4max}  (see also \cite[Table 4.7.3A]{GLS}) there are two $G_{0}^{\ast}$-classes of graph automorphisms and
 either $C_{G_{0}}(x) \cong  G_{2}(q)$ or $  C_{G_{0}}(x) \cong \PGL_{3}^{\varepsilon}(q)$,  where $\varepsilon$ is determined by the congruence $q\equiv  \varepsilon  \pmod{3}$. Observe that in  both cases $C_{G_{0}}(x)$ is almost simple provided $q \ne 2$, and in particular  has a trivial centre. Moreover, $3$ divides $|G_{0}:G_{2}(q)  |$ and $|G_{0}:\PGL_{3}^{\varepsilon}(q)  |$. Thus we obtain a  contradiction to the minimality of $(G,x)$ by Lemma \ref{Lem_CentralizerCosetArgument} for $q \ne 2$. For $q=2$ we have $G_0 \in CVL3$.

 (B.3.2) $G_0 \cong D_{4}(3^{a})$ or ${^3}D_{4}(3^a)$.  By  \cite[Proposition 4.9.2(g)]{GLS}, $x$ is $G_{0}^{\ast}$-conjugate to $\gamma ^{\pm 1}$ or $(\gamma z)^{\pm 1}$ where $\gamma \in \Gamma _{G_{0}}$ if $G_{0}\cong D_{4}(q)$ and $\gamma \in \Phi _{G_{0}}$ if $G_{0}\cong {}^{3}D_{4}(q)$, and $  z $ is a nontrivial element of a long root subgroup $Z=C_{\overline{X}  _{-\alpha _{\ast}}}(\sigma ) \leq G_{0}$, where $\sigma $ is  the Steinberg endomorphism used for the $\sigma $-setup defining $G_{0}$ (\cite[Example 3.2.6]{GLS}). Clearly, we may assume that $x= \gamma$ or $\gamma z$.  If $x=\gamma $, then, by  \cite[Proposition 4.9.2(b5)]{GLS}, we have $C_{G_{0}}(x) \cong G_{2}(q)$ and we obtain  a contradiction by Lemma \ref{Lem_CentralizerCosetArgument} as in (B.3.1). Now suppose that $x=\gamma z$. By \cite[Proposition 4.9.2]{GLS},   we have  $C_{G_{0}}(x)  =C_{C_{G_{0}}(\gamma )}(z)$. Now  $z\in C_{G_{0}}(\gamma ) \cong G_2(q)$ and $G_2(q)$ has trivial centre,  thus $C_{G_{0}}(x)  =C_{C_{G_{0}}(\gamma )}(z) \lneqq  G_2(q) \cong C_{G_{0}}(\gamma)$. However, since $\gamma $ and $z$ commute, the element $x=\gamma z$ normalizes $C_{G_{0}}(\gamma ) \cong G_{2}(q)$. Now $C_{G_{0}}(\gamma )$ is not centralized by $z$, so
  $x=\gamma z$ normalizes but does not centralize $C_{G_{0}}(\gamma ) \cong G_{2}(q)$, which contradicts the minimality of $(G,x)$.

 (B.4) \emph{Graph-field automorphisms.} Suppose that $x$ is a graph-field automorphism. Arguments similar to those used in  (B.4) in Section \ref{Sect_ProofFirstTho(x)=2},
 eliminate all cases except the case $G_{0}=D_{4}(q_0^{3})$. By \cite[Proposition 4.9.1]{GLS} we have $O^{r^{\prime}}(C_{G_{0}}(x)) \cong {}^{3}D_{4}(q_{0})$ and  $C_{G_{0}^{\ast}}(x) \cong ( O^{r^{\prime}}(C_{G_{0}}(x) ) )^{\ast} \cong {^3}D_{4}(q_{0})$. 
Hence $C_{G_{0}^{\ast}}(x)$ is simple and since $C_{G_{0}}(x) \trianglelefteq C_{G_{0}^{\ast}}(x)$, we have $C_{G_{0}}(x) =C_{G_{0}^{\ast}}(x) \cong{^3}D_{4}(q_{0})$. Now $3$ divides $| D_{4}(q):{^3}D_{4}(q_{0})|$ and ${^3}D_{4}( q_{0})$ has trivial centre, so we may apply Lemma \ref{Lem_CentralizerCosetArgument} to eliminate this case.

(C) \emph{Sporadic groups.}

  We may assume that $G=G_{0}$ in all cases since $| \mathrm{Out}(G_{0})| =1$ or $2$ and $x$ has order $3$. In order to eliminate  a given $G_{0}$, it is sufficient to show that $x$ belongs to an almost simple subgroup $H \le G$. This is immediate  when $G_{0}$ has a single conjugacy class of elements of order $3$ and there  exists almost simple subgroup $H$ of  order divisible by $3$. More  generally, we can eliminate a conjugacy class $C$ of order $3$ elements if there exists a unique conjugacy class of elements of order $3k$ that powers up to $C$ and there exists an almost simple subgroup $H$ that contains elements of order $3k$.   
   We use \cite{Atlas,wwwatlas} to do this. When  such an argument cannot be established, we use \textsc{Magma}.
\end{proof}


\end{document}